\documentclass[11pt,twoside,a4paper]{article}
\usepackage{amsmath,amssymb,amsthm}
\usepackage{hyperref}
\usepackage{graphicx,psfrag}

\newtheorem{thm}{Theorem}[section]
\newtheorem{lem}[thm]{Lemma}
\newtheorem{pro}[thm]{Proposition}
\newtheorem{cor}[thm]{Corollary}

\theoremstyle{definition}
\newtheorem{exa}[thm]{Example}

\theoremstyle{remark}
\newtheorem{rem}[thm]{Remark}

\newcommand{\R}{\mathbb{R}}
\newcommand{\Z}{\mathbb{Z}}
\newcommand{\N}{\mathbb{N}}

\newcommand{\cI}{\mathcal{I}}

\newcommand{\cM}{\mathcal{M}}

\newcommand{\cP}{\mathcal{P}}

\newcommand{\al}{\alpha}
\newcommand{\be}{\beta}
\newcommand{\ga}{\gamma}

\newcommand{\de}{\delta}
\newcommand{\De}{\Delta}
\newcommand{\ep}{\varepsilon}
\newcommand{\om}{\omega}

\newcommand{\si}{\sigma}

\renewcommand{\phi}{\varphi}

\newcommand{\dist}{\operatorname{dist}}
\newcommand{\diam}{\operatorname{diam}}

\newcommand{\CAT}{\operatorname{CAT}}
\newcommand{\hyp}{\operatorname{H}}
\newcommand{\id}{\operatorname{id}}

\newcommand{\pr}{\operatorname{pr}}

\newcommand{\crr}{\operatorname{cr}}

\newcommand{\reg}{\operatorname{reg}}
\renewcommand{\d}{\partial}
\newcommand{\di}{\d_{\infty}}

\newcommand{\ay}{\operatorname{aY}}
\newcommand{\harm}{\operatorname{Harm}}
\newcommand{\hm}{\operatorname{Hm}}

\newcommand{\h}{\operatorname{h}}
\newcommand{\co}{\operatorname{co}}
\newcommand{\width}{\operatorname{width}}

\newcommand{\es}{\emptyset}
\newcommand{\set}[2]{\{#1:\,\text{#2}\}}
\newcommand{\sm}{\setminus}
\newcommand{\sub}{\subset}

\newcommand{\ov}{\overline}
\newcommand{\wt}{\widetilde}
\newcommand{\wh}{\widehat}

\begin{document}

\title{Inverse problem for M\"obius geometry on the circle}
\author{Sergei Buyalo\footnote{This work is supported by the RFBR grant 17-01-00128 and by
Presidium of RAS program ``New methods of mathematical modelling in non-linear dynamical systems'' 
(grant 08-04)}}

\date{}
\maketitle

 \begin{abstract} We give a solution to the inverse problem of 
M\"obius geometry on the circle. Namely, we describe a class of M\"obius
structures on the circle for each of which there is a hyperbolic space 
such that its boundary at infinity is the circle, and the induced M\"obius
structure coincides with the given one. That class is not empty and form
an open neighborhood of the canonical M\"obius structure in an appropriate
fine topology.
 \end{abstract}


\medskip


\section{Introduction}
\label{introduction}

Any (boundary continuous) hyperbolic space induces on the boundary at infinity 
a M\"obius structure which reflects most essential asymptotic properties of the space.
A {\em M\"obius structure}
$M$
on a set 
$X$
is a class of M\"obius equivalent semi-metrics on
$X$,
where two semi-metrics are equivalent if and only if they have
the same cross-ratios on every 4-tuple of points in 
$X$.
In other words, a M\"obius structure is given by cross-ratios.

The {\em inverse problem} of M\"obius geometry asks to describe M\"obius structures which are
induced by hyperbolic spaces. In this paper, we give a solution to the inverse problem
in a simplest case when the space
$X$
is the circle,
$X=S^1$.
The paper is a continuation of \cite{Bu18}, \cite{Bu19}, where the inverse problem is formulated,
and important steps toward its solution are done.

Various hyperbolic cone constructions (see \cite{BoS}, \cite{BS07}) give a hyperbolic metric space with
prescribed metric at infinity. However, no one of them is equivariant with respect to M\"obius
transformations of the metric. Thus one can consider the inverse problem as the existence problem
of an equivariant hyperbolic cone over a given metric.

We introduce a set of axioms describing M\"obius structures on the circle, which are
induced by hyperbolic spaces. We always consider {\em ptolemaic} M\"obius structures,
that is, for which every semi-metric with infinitely remote point is a metric. Our 
{\em monotonicity} axiom is somewhat stronger than that in \cite{Bu18}. 
Thus a M\"obius structure, which satisfies it, is called {\em strictly monotone}. 
As in \cite{Bu18}, we also use a key {\em Increment} axiom. For the definition and details see 
sect.~\ref{subsect:increment_axiom}.

The main result of the paper is the following

\begin{thm}\label{thm:main} Given a strictly monotone M\"obius structure
$M$
on the circle satisfying Increment axiom, there is a complete, 
proper and geodesic hyperbolic metric space
$Y$
with boundary at infinity
$\di Y=S^1$,
for which the induced M\"obius structure
$M_Y$
on
$\di Y$
is isomorphic to 
$M$, $M_Y=M$.
\end{thm}

\begin{rem}\label{rem:fine_topology} The class
$\cI$
of strictly monotone M\"obius structures on the circle which satisfy Increment
axiom contains an open in a fine topology neighborhood of the canonical
M\"obius structure
$M_0$,
see sect.~\ref{subsect:increment_axiom}. 
\end{rem}

{\em Structure of the paper}. In section~\ref{sect:moebius_structures}, we give 
a brief introduction to M\"obius structures, formulate basic axioms, including
Increment axiom, and discuss a fine topology on the set 
$\cM$
M\"obius structures satisfying our axioms.

In section~\ref{sect:lines_zzpath} we recall the notions of lines and zz-paths
associated with a given M\"obius structure
$M\in\cM$. 
After a brief discussion in sect.~\ref{sect:metric}
of the metric on the set 
$\harm$
of harmonic 4-tuples, we consider in 
sect.~\ref{sect:involutions_without_fixed_points} an important notion of involutions
without fixed points and the associated notion of elliptic quasi-lines. Given
$\om\in X=S^1$,
we consider here the set 
$\harm_\om\sub\harm$
of harmonic 4-tuples containing
$\om$.
Such sets play a important role in the proof of the main theorem.

A key technical part of the paper is section~\ref{sect:diameter_quasi_lines},
where we give an universal upper bound for the diameter of elliptic
quasi-lines. Such estimate allows to reduce the study of geometry on
the space
$\harm$
to the study of its much simpler subspaces
$\harm_\om$.

In section~\ref{sect:hyperbolic_approximations}, we discuss properties of
a hyperbolic cone construction over
$X_\om$
called the hyperbolic approximation 
$Z$
of
$X_\om$.
We show here that
$Z$
is a hyperbolic geodesic metric space. This section is based on the book \cite[Chapter~6]{BS07}.

Finally, in sect.~\ref{sect:Xquasi-isometricZ} we show that the spaces
$\harm_\om$
and
$Z$
are quasi-isometric. As a corollary, we obtain that the required filling 
$Y=\harm$
of a given M\"obius structure
$M\in\cM$
on the circle is hyperbolic. The proof essentially uses Increment axiom and results
of \cite{Bu18}.

\section{M\"obius structures}
\label{sect:moebius_structures}

\subsection{Basic notions}
\label{subsect:basics}

Let
$X$
be a set. A 4-tuple
$q=(x,y,z,u)\in X^4$
is said to be {\em admissible} if no entry occurs three or
four times in
$q$.
A 4-tuple
$q$
is {\em nondegenerate}, if all its entries are pairwise
distinct. Let
$\cP_4=\cP_4(X)$
be the set of all ordered admissible 4-tuples of
$X$, $\reg\cP_4\sub\cP_4$
the set of nondegenerate 4-tuples.

A function
$d:X^2\to\wh\R=\R\cup\{\infty\}$
is said to be a {\em semi-metric}, if it is symmetric,
$d(x,y)=d(y,x)$
for each
$x$, $y\in X$,
positive outside the diagonal, vanishes on the diagonal
and there is at most one infinitely remote point
$\om\in X$
for
$d$,
i.e. such that
$d(x,\om)=\infty$
for some
$x\in X\sm\{\om\}$.
Moreover, we require that if
$\om\in X$
is such a point, then
$d(x,\om)=\infty$
for all 
$x\in X$, $x\neq\om$.
A metric is a semi-metric that satisfies the triangle inequality.

A {\em M\"obius structure}
$M$
on
$X$
is a class of M\"obius equivalent semi-metrics on
$X$,
where two semi-metrics are equivalent if and only if they have
the same cross-ratios on every
$q\in\reg\cP_4$.

Given
$\om\in X$,
there is a semi-metric 
$d_\om\in M$
with infinitely remote point
$\om$.
It can be obtained from any semi-metric
$d\in M$
for which 
$\om$
is not infinitely remote by a {\em metric inversion},
\begin{equation}\label{eq:metric_inversion}
d_\om(x,y)=\frac{d(x,y)}{d(x,\om)d(y,\om)}. 
\end{equation}
Such a semi-metric is unique up to a homothety, see \cite{FS13},
and we use notation
$|xy|_\om=d_\om(x,y)$
for the distance between
$x$, $y\in X$
in that semi-metric. We also use notation
$X_\om=X\sm\{\om\}$.

Every M\"obius structure
$M$
on
$X$
determines the 
$M$-{\em topology}
whose subbase is given by all open balls centered at finite points
of all semi-metrics from
$M$
having infinitely remote points.

\begin{exa}\label{exa:canonical_moebius_circle} Our basic example is the 
{\em canonical} M\"obius structure 
$M_0$
on the circle
$X=S^1$.
We think of
$S^1$
as the unit circle in the plane,
$S^1=\set{(x,y)\in\R^2}{$x^2+y^2=1$}$.
For 
$\om=(0,1)\in X$
the stereographic projection
$X_\om\to\R$
identifies
$X_\om$
with real numbers 
$\R$.
We let
$d_\om$
be the standard metric on
$\R$,
that is,
$d_\om(x,y)=|x-y|$
for any
$x,y\in\R$.
This generates a M\"obius structure on
$X$
which is called {\em canonical}. The basic feature of the canonical M\"obius 
structure on
$X=S^1$
is that for any 4-tuple
$(\si,x,y,z)\sub X$
with the cyclic order 
$\si xyz$
we have 
$d_\si(x,y)+d_\si(y,z)=d_\si(x,z)$.
\end{exa}

\subsection{Harmonic pairs}
\label{subsect:harm_pairs}

From now on, we assume that 
$X$
is the circle,
$X=S^1$.
It is convenient to use unordered pairs 
$(x,y)\sim(y,x)$
of distinct points on
$X$,
and we denote their set by
$\ay=S^1\times S^1\sm\De/\sim$,
where
$\De=\set{(x,x)}{$x\in S^1$}$
is the diagonal. A pair 
$q=(a,b)\in\ay\times\ay$
is harmonic if
\begin{equation}\label{eq:harmonic}
|xz|\cdot|yu|=|xu|\cdot|yz|
\end{equation}
for some and hence any semi-metric of the M\"obius structure, where
$a=(x,y)$, $b=(z,u)$.
The pair
$a$
is called the {\em left} axis of
$q$,
while
$b$
the {\em right} axis. We denote by
$\harm$
the set of harmonic pairs,
$\harm\sub\ay\times\ay$,
of the given M\"obius structure. There is a canonical involution 
$j:\harm\to\harm$
without fixed points given by
$j(a,b)=(b,a)$.
Note that
$j$
permutes left and right axes. The quotient space we denote by
$\hm:=\harm/j$.
In other words,
$\hm$
is the set of unordered harmonic pairs of unordered pairs 
of points in
$X$,
and
$\harm$
is its 2-sheeted covering.
\begin{rem}\label{rem:2-covering_harm} Sometimes, we need a 2-sheeted
covering
$\wt\harm$
of
$\harm$,
which consists of harmonic pairs
$q=(a,b)$
with
$a=(x,y)\in S^1\times S^1\sm\De$, $b\in\ay$. 
Note that
$\wt\harm$
is homeomorphic to the tangent bundle of
$\hyp^2$.
\end{rem}

 \subsection{Axioms}
\label{subsect:axioms}

We list a set of axioms for a M\"obius structure
$M$ 
on the circle
$X=S^1$,
which needed for Theorem~\ref{thm:main}.

\begin{itemize}
 \item [(T)] Topology: $M$-topology
on
$X$
is that of
$S^1$.

\item[(M($\al$))] Monotonicity: Fix 
$1>\al\ge\sqrt{2}-1$.
Given a 4-tuple
$q=(x,y,z,u)\in X^4$
such that the pairs
$(x,y)$, $(z,u)$
separate each other, we have

$$|xy|\cdot|zu|\ge\max\{|xz|\cdot|yu|+\al|xu|\cdot|yz|,\al|xz|\cdot|yu|+|xu|\cdot|yz|\}$$
for some and hence any semi-metric from
$M$.

\item[(P)] Ptolemy: for every 4-tuple
$q=(x,y,z,u)\in X^4$
we have
$$|xy|\cdot|zu|\le |xz|\cdot|yu|+|xu|\cdot|yz|$$
for some and hence any semi-metric from
$M$.
\end{itemize}

A M\"obius structure 
$M$
on the circle
$X$
that satisfies axioms T, M($\al$), P is said to be {\em strictly monotone}.
We denote by
$\cM$
the class of strictly monotone M\"obius structures on
$X$.

\begin{rem}\label{rem:zolotov} Axiom~M($\al$) is motivated by the work
\cite{Zo18} of V.~Zolotov. It is stronger than that in
\cite{Bu19}. The lower bound for 
$\al$
is used in sect.~\ref{subsect:diam_quai-lines}.
\end{rem}

\begin{rem}\label{rem:axiom_Q} Axiom~P is satisfied, for example,
for the M\"obius structure on the boundary at infinity of any
$\CAT(-1)$
space, see \cite{FS12}.
\end{rem}

\begin{rem}\label{rem:canonical_axioms} The canonical M\"obius structure
$M_0$
on
$X=S^1$
clearly satisfies Axioms~T, M($\al$), P.
\end{rem}

We recall some immediate corollaries from the axioms, see \cite{Bu19}. It follows from axiom~(P) 
that any semi-metric from
$M$
with an infinitely remote point is a metric, i.e. it satisfies the triange inequality.

A choice of
$\om\in X$
uniquely determines the interval
$xy\sub X_\om$
for any distinct
$x$, $y\in X$
different from
$\om$
as the arc in
$X$
with the end points
$x$, $y$
that does not contain
$\om$.

We have \cite[Corollary~2.6, Corollary~2.7 ]{Bu19}.

\begin{cor}\label{cor:interval_monotone} 
Axiom M($\al$) implies the following. Assume for a nondegenerate 4-tuple
$q=(x,y,z,u)\in\reg\cP_4$
the interval
$xz\sub X_u$
is contained in
$xy$, $xz\sub xy\sub X_u$.
Then
$|xz|_u<|xy|_u$.
\end{cor}

\begin{cor}\label{cor:harm_separate} For any harmonic pair
$((x,y),(z,u))\in\harm$
the pairs
$(x,y)$, $(z,u)\in\ay$
separate each other.
\end{cor}

\subsection{Increment axiom and a fine topology on $\cM$}
\label{subsect:increment_axiom}

Increment axiom is not used explicitly in the paper.
However, it is very important in proving that lines with respect to
a M\"obius structure are geodesic, see \cite{Bu18}. We recall it here
for convenience of the reader. For more details see \cite{Bu17}, where
it has been introduced.

The following is an alternative description of a M\"obius structure which
is convenient in many cases. For any semi-metric
$d$
on
$X$
we have three cross-ratios
$$q\mapsto \crr_1(q)=\frac{|x_1x_3||x_2x_4|}{|x_1x_4||x_2x_3|};
  \crr_2(q)=\frac{|x_1x_4||x_2x_3|}{|x_1x_2||x_3x_4|};
  \crr_3(q)=\frac{|x_1x_2||x_3x_4|}{|x_2x_4||x_1x_3|}$$
for 
$q=(x_1,x_2,x_3,x_4)\in\reg\cP_4$,
whose product equals 1, where
$|x_ix_j|=d(x_i,x_j)$.
We associate with 
$d$
a map 
$M_d:\reg\cP_4\to L_4$
defined by
\begin{equation}\label{eq:moeb_map}
M_d(q)=(\ln\crr_1(q),\ln\crr_2(q),\ln\crr_3(q)),
\end{equation}
where
$L_4\sub\R^3$
is the 2-plane given by the equation
$a+b+c=0$.
Two semi-metrics
$d$, $d'$
on
$X$
are M\"obius equivalent if and only
$M_d=M_{d'}$.
Thus a M\"obius structure on
$X$
is completely determined by a map 
$M=M_d$
for any semi-metric
$d$
of the M\"obius structure, and we often identify a M\"obius structure
with the respective map 
$M$.

In this description, axioms~(M($\al$)) and (P) are these:

M(($\al$)) Fix 
$1>\al\ge\sqrt{2}-1$.
Given a 4-tuple
$q=(x,y,z,u)\in X^4$
such that the pairs
$(x,y)$, $(z,u)$
separate each other, we have
$$\crr_3(q)\ge\max\left\{1+\frac{\al}{\crr_1(q)},\al+\frac{1}{\crr_1(q)}\right\}.$$

(P) for every 4-tuple
$q=(x,y,z,u)\in X^4$
we have
$$\crr_3(q)\le 1+\frac{1}{\crr_1(q)}.$$

We use notation
$\reg\cP_n$
for the set of ordered nondegenerate
$n$-tuples
of points in
$X=S^1$, $n\in\N$.
For 
$q\in\reg\cP_n$
and a proper subset
$I\sub\{1,\dots,n\}$
we denote by
$q_I\in\reg\cP_k$, $k=n-|I|$,
the 
$k$-tuple
obtained from
$q$
(with the induced order) by crossing out all entries which correspond to elements of
$I$.

(I) Increment Axiom: for any 
$q\in\reg\cP_7$
with cyclic order
$\co(q)=1234567$
such that 
$q_{247}$
and 
$q_{157}$
are harmonic, we have 
$$\crr_1(q_{345})>\crr_1(q_{123}).$$

It is proved in \cite[Proposition~7.10]{Bu17} that the canonical M\"obius
structure
$M_0$
on the circle
$X=S^1$
satisfies Increment Axiom. 

We define a fine topology on
$\cM$
as follows. Let
$\reg^+\cP_7\sub X^7$
be the subset of
$\reg\cP_7$
which consists of all 
$q\in\reg\cP_7$
with the cyclic order. We take on
$\reg^+\cP_7$
the topology induced from the standard topology of the 7-torus
$X^7$.
We associate with a M\"obius structure
$M\in\cM$
a section of the trivial bundle
$\reg^+\cP_7\times\R^4\to\reg^+\cP_7$
given by
$$M(q)=(q,\crr_2(q_{247}),\crr_2(q_{157}),\crr_1(q_{345}),\crr_1(q_{123}))$$
for 
$q=1234567\in\reg^+\cP_7$.
Taking the product topology on
$\reg^+\cP_7\times\R^4$,
we define the {\em fine} topology on
$\cM$
with base given by sets
$$U_V=\set{M\in\cM}{$M(\reg^+\cP_7)\sub V$},$$
where
$V$
runs over open subsets of
$\reg^+\cP_7\times\R^4$.

The class
$\cI$
of (strictly) monotone M\"obius structures on the circle which satisfy Axiom~(I)
contains an open in the fine topology neighborhood of 
$M_0$,
see \cite[Proposition~7.14]{Bu17}. 

\section{Lines and zigzag paths}
\label{sect:lines_zzpath}

Here we briefly recall definitions and some properties of lines and zigzag paths
from \cite{Bu18}, \cite{Bu19}.

\subsection{Lines}
\label{subsect:lines}

\begin{lem}\label{lem:project_point_line}\cite[Lemma~3.1]{Bu19} Given
$a\in\ay$
and
$x\in X$, $x\notin a$,
there is a uniquely determined
$y\in X$
such that the pair
$(a,b)$
is harmonic, 
$(a,b)\in\hm$,
where
$b=(x,y)$.
\end{lem}

We denote by
$\rho_a(x)=y$
the point
$y$
from Lemma~\ref{lem:project_point_line}. The {\em line} with axis
$a\in\ay$
is defined as the set  
$\h_a\sub\hm$
which consists of all pairs
$q=(a,b)$
with
$b=(x,\rho_a(x))$
where 
$x$
run over an arc in
$X$
determined by
$a$.
This is well defined because
$\rho_a:X\to X$
is involutive,
$\rho_a^2=\id$
(we extend
$\rho_a$
to
$a=(z,u)$
by
$\rho_a(z)=z$, $\rho_u=u$). In this case, we use notation
$x_a:=b$
and say that
$x_a\in\h_a$
is the projection of
$x$
to the line
$\h_a$.

For more about lines see \cite{Bu18}. In partial, every line 
is homeomorphic to the real line 
$\R$,
different points on a line are in {\em strong causal relation}, that is,
either of them lies on an open arc in
$X$
determined by the other one, and vice versa, given 
$b$, $b'\in\ay$
in strong causal relation, there exists a unique line
$\h_a$
through
$b$, $b'$,
see \cite[Lemma~3.2, Lemma~4.2]{Bu18}. In this case, the pair
$a\in\ay$ 
(or the line
$\h_a$)
is called the {\em common perpendicular} to
$b$, $b'$.

The {\em segment}
$qq'$
of a line
$\h_a$
with
$q=(a,b)$, $q'=(a,b')\in\h_a$
is defined as the union of 
$q$, $q'$
and all 
$q''=(a,b'')\in\h_a$
such that 
$b''$
separates
$b$, $b'$.
The last means that 
$b$
and
$b'$
lie on different open arcs in
$X$
determined by
$b''$.
The points
$q$, $q'$
are the {\em ends} of
$qq'$.
The segment
$qq'\sub\h_a$
is homeomorphic to the standard segment
$[0,1]$.

\subsection{Distance between harmonic pairs with common axis}
\label{subsect:distance_harmonic_pairs}

Given two harmonic pairs in
$q$, $q'\in\hm$
with a common axis, say
$q=(a,b)$
and
$q'=(a,b')$,
we define {\em the distance}
$|qq'|$
between them as
\begin{equation}\label{eq:distance}
|qq'|=\left|\ln\frac{|xz'|\cdot|yz|}{|xz|\cdot|yz'|}\right|
\end{equation}
for some and hence any semi-metric on
$X$
from
$M$,
where
$a=(x,y)$, $b=(z,u)$, $b'=(z',u')\in\ay$.

One easily checks that every line
$\h_a\sub\hm$
with this distance is isometric to the real line
$\R$
with the standard distance.

\subsection{Zigzag paths}
\label{subsect:zigzag_paths}

Every harmonic pair
$q=(a,b)\in\hm$
has two axes. Thus moving along of a line, we have a possibility
to change the axis of the line at any moment and move along the line
determined by the other axis. This leads to the notion of zig-zag path.
A {\em zig-zag} path, or zz-path, 
$S\sub\hm$
is defined as finite (maybe empty) sequence of segments 
$\si_i$
in
$\hm$,
where consecutive segments 
$\si_i$, $\si_{i+1}$
have a common end
$q=\si_i\cap\si_{i+1}\in\hm$
with axes determined by
$\si_i$, $\si_{i+1}$.
Segments 
$\si_i$
are also called {\em sides} of
$S$,
while a {\em vertex} of
$S$
is an end of a side. Given 
$q$, $q'\in\hm$,
there is a zz-path
$S$
in
$\hm$
with at most five sides that connects
$q$
and
$q'$
(see \cite[Lemma~3.3]{Bu18}). This notion is easily lifted to
$\harm$.

\section{Metric on $\hm$ and filling of $M$}
\label{sect:metric}

\subsection{Distance $\de$ on $\hm$}
\label{subsect:dist_de}

Let
$S=\{\si_i\}$
be a zz-path in
$\hm$.
We define the length of
$S$
as the sum
$|S|=\sum_i|\si_i|$
of the length of its sides. Now, we define a distance
$\de$
on
$\hm$
by
$$\de(q,q')=\inf_S|S|,$$
where the infimum is taken over all zz-paths
$S\sub\hm$
from
$q$
to
$q'$.

One easily sees that 
$\de$
is a finite pseudometric on
$\hm$,
see \cite[Proposition~6.2]{Bu18}.
The following result is obtained in \cite{Bu18}, \cite{Bu19}.

\begin{thm}\label{thm:de_metric_space} Assume that a M\"obius structure 
$M$
on
$X=S^1$
is strictly monotone, i.e., it satisfies axioms~(T), (M($\al$)), (P). Then
$(\hm,\de)$
is a complete, proper, geodesic metric space with
$\de$-metric
topology coinciding with that induced from
$X^4$. If, in addition
$M$
satisfies Increment axiom, then every line in
$\hm$
is a geodesic.
\end{thm}

\begin{rem}\label{rem:hm_vs_harm} Since
$\harm$
is a 2-sheeted covering of
$\hm$,
all of the conclusions of Theorem~\ref{thm:de_metric_space} hold for the space
$\harm$.
\end{rem}

\subsection{Filling}
\label{subsect:filling}

Now we define a filling 
$Y$
of a strictly monotone M\"obius structure
$M$
on
$X$
as the space 
$(\hm,\de)$
of harmonic pairs in
$M$
with the distance
$\de$, $Y=(\hm,\de)$.
Our aim is to show under the assumption that 
$M$
in addition satisfies Increment axiom
$Y$
is a required in Theorem~\ref{thm:main} hyperbolic space. Sometimes, we pass to
its 2-sheeted covering
$\harm$
and use the same notation
$Y=(\harm,\de)$.

\section{Involutions of $X$ without fixed points}
\label{sect:involutions_without_fixed_points}

\subsection{Some properties}
\label{subsect:properties_fixed_points}

Involution
$\rho:X\to X$
of
$X=S^1$
is an involutive,
$\rho^2=\id$,
homeomorphism.

\begin{lem}\label{lem:separate} Let
$\rho:X\to X$
be an involution without fixed points. Then for any distinct
$x$, $y\in X$
the pairs
$a=(x,\rho(x))$, $b=(y,\rho(y))$
separate each other.
\end{lem}

\begin{proof} Assume to the contrary that there are distinct
$x$, $y\in X$
such that the respective
$a$, $b\in\ay$
do not separate each other. Let
$X=a^+\cup a^-$
decomposition of
$X$
into (closed) arcs determined by
$a$.
By the assumption,
$b$ 
lies on one of these arcs, say
$b\sub a^+$.
Since
$\rho$
is an involution, we have
$\rho(a)=a$
and
$\rho(b)=b$.
Therefore,
$\rho$
preserves 
$a^+$
permuting its ends
$x$, $\rho(x)$.
But in this case we observe a fixed point of
$\rho$
inside of
$a^+$.
This is a contradiction because
$\rho$
has no fixed points.
\end{proof}

Let
$\rho:X\to X$
be an involution without fixed points. The factor
$X/\rho$
can be identified with the subset
$$e_\rho=\set{(x,\rho(x))\in\ay}{$x\in X$}\sub\ay,$$
which is called an {\em elliptic quasi-line.}

\begin{lem}\label{lem:harmonic_pairs_ellitic} Let
$e=e_\rho$
be an elliptic quasi-line in
$\ay$.
Then for every
$s\in\ay$
there is a unique
$t\in e$
such that the 4-tuple
$(s,t)$
is harmonic.
\end{lem}

\begin{proof} First, we show that the image under the involution
$\rho$
of at least one of the open arcs
$s^+$, $s^-$, 
in which
$s=(x,y)$
separates
$X$,
misses that arc. Indeed, if
$\rho(x)=y$,
then
$\rho(y)=x$.
In that case, 
$\rho$
permutes the arcs
$s^+$, $s^-$
since otherwise,
$\rho(s^\pm)=s^\pm$,
and thus
$\rho$
has a fixed point.

By Lemma~\ref{lem:separate} we know that the pairs
$(x,\rho(x))$
and
$(y,\rho(y))$
separate each other. Hence,
$\rho(s)$
and
$s$
do not separate each other, and we can assume without loss of generality, that
$\rho(s)\sub s^-$.
Then
$\rho(s^+)$
misses
$s^+$
since otherwise
$\rho(s^+)\supset s^+$,
and thus
$\rho$
has a fixed point.

We denote that arc by
$s^+$
and define a function 
$f:s^+\to\R$
by
$$f(z)=\frac{|zy|_x}{|\rho(z)y|_x},$$
where recall 
$x$
is the infinitely remote point for the semi-metric
$|zu|_x$.
By the choice of
$s^+$,
we have
$\rho(z)=y$
for no
$z\in s^+$.
Thus
$f$
is continuous, 
$f(z)\to\infty$
as
$z\to x$
and
$f(z)\to 0$
as
$z\to y$.
By continuity,
$f(z)=1$
for some
$z\in s^+$.
Then the 4-tuple
$(s,t)$
is harmonic for 
$t=(z,\rho(z))\in e$.

If
$t'\in e$
is another element with harmonic
$(s,t')$,
then 
$s$
is the common perpendicular to
$t$, $t'$
and thus
$t$, $t'$
are in the strong causal relation see sect.~\ref{subsect:lines}, in particular, they  
do not separate each other. This contradicts the conclusion of
Lemma~\ref{lem:separate}.
\end{proof}

\begin{rem}\label{rem:lift_quasi-line}
Let
$\rho:X\to X$
be an involution without fixed points. Applying Lemma~\ref{lem:harmonic_pairs_ellitic}
to any
$s\in e_\rho$
we obtain a harmonic pair
$(s,t(s))\in\harm$
with both
$s$, $t(s)\in e_\rho$.
The set 
$\wh e_\rho=\set{(s,t(s))}{$s\in e_\rho$}\sub\harm$
is also called the {\em elliptic quasi-line} in
$\harm$
associated with the involution
$\rho$.
In this sense, we can lift any elliptic quasi-line
$e_\rho\sub\ay$
to the uniquely determined elliptic quasi-line
$\wh e_\rho\sub\harm$.
It follows from Lemma~\ref{lem:harmonic_pairs_ellitic} and Lemma~\ref{lem:separate}
that 
$\wh e_\rho$ 
is invariant under the involution
$j:\harm\to\harm$.
Thus we can speak about elliptic quasi-lines in
$\hm$.
\end{rem}

\subsection{Involutions associated with a harmonic 4-tuple}
\label{subsect:involution_harm}

Every harmonic 4-tuple
$q=(a,b)\in\harm$
generates a pair of involutions
$\rho_q^\pm:X\to X$
without fixed points as follows. We fix decomposition of 
$X\sm a$
into open arcs
$a^\pm$
with the common ends
$a$, $X=a^+\cup a^-\cup a$,
and define maps
$\rho_q^\pm:X\to X$
by
$$\rho_q^\pm(x)=\begin{cases}
          \rho_b\circ\rho_a(x),\ x\in\ov a^\pm\\
          \rho_a\circ\rho_b(x),\ x\in\ov a^\mp,
         \end{cases}
$$
where
$\ov a^\pm$
are respective closed arcs. Since
$\rho_b\circ\rho_a(x)=\rho_a\circ\rho_b(x)$
for 
$x=a$,
the maps 
$\rho_q^\pm$
are well defined and they are continuous involutions of
$X$
without fixed points. Since
$\rho_a(b)=b$
and
$\rho_b(a)=a$,
it follows from Lemma~\ref{lem:harmonic_pairs_ellitic} that
$q\in\wh\rho_\rho$
for 
$\rho=\rho_q^\pm$.

\begin{rem}\label{rem:noncommuting} The maps 
$\rho_a$, $\rho_b$
may not be commuting, thus
$\rho^+\neq\rho^-$
in general, and to define an involution
$\rho$
we are forced to make a choice of one of the arcs, in which
$a$ 
(or
$b$)
separates
$X$. 
\end{rem}

\subsection{Canonical decomposition of $\harm$ over $X$}
\label{subsect:canonical_decomposition}

For every
$\om\in X$
consider the set 
$\harm_\om$
which consists of all pairs 
$q=(a,b)\in\harm$
with
$\om\in a$.
Clearly,
$\harm=\cup_{\om\in X}\harm_\om$,
and for different
$\om$, $\om'\in X$
the sets
$\harm_\om$, $\harm_{\om'}$
intersect over the line
$h_{(\om,\om')}$, $\harm_\om\cap\harm_{\om'}=h_{(\om,\om')}$.

Our aim in this section is to show that every
$\harm_\om$
is cobounded in
$\harm$
uniformly in
$\om\in X$,
see Corollary~\ref{cor:uniform_cobounded}.

\subsection{Virtual projection $\harm\to\harm_\om$}
\label{subsect:construction_harm_to_harm_omega}

Involutions associated with
$q=(a,b)\in\harm$
depend on the choice of arcs
$a^+$, $a^-$,
see sect.~\ref{subsect:involution_harm}. To make that choice canonical,
we fix an orientation of the circle
$X=S^1$
and pass to the 2-sheeted covering
$\wt\harm$
of
$\harm$,
see Remark~\ref{rem:2-covering_harm}. Then for every
$q=(a,b)\in\wt\harm$, $a=(x,y)\in X^2$,
the arc 
$a^+$
is defined as the oriented arc from
$x$
to
$y$
with the orientation induced by the orientation of
$X$.
Now, we define
$\rho_q=\rho^+_q$.

\begin{lem}\label{lem:omega_projection} For every
$\om\in X$
there is a well defined retraction
$h_\om:\wt\harm\to\harm_\om$.
\end{lem}

\begin{proof} Given
$q=(a,b)\in\wt\harm$
we consider the quasi-elliptic line
$e=e_\rho$
associated with the involution
$\rho=\rho^+_q:X\to X$.
Then the line
$h_s\sub\harm$
with
$s=(\om,\rho(\om))\in\ay$
lies in fact in
$\harm_\om$
by the definition, 
$h_s\sub\harm_\om$.
By Lemma~\ref{lem:harmonic_pairs_ellitic}, there is a uniquely determined
$t\in e$
with
$(s,t)$
harmonic, that is,
$(s,t)\in h_s$.
Now, we put
$h_\om(q)=(s,t)$.
This canonically defines a retraction
$h_\om:\wt\harm\to\harm_\om$
which we call a {\em virtual} projection of
$\harm$
to
$\harm_\om$.
\end{proof}

\section{Diameter of elliptic quasi-lines}
\label{sect:diameter_quasi_lines}

In this section, we show that the diameter of any elliptic quasi-line in
$\harm$
is uniformly bounded above.

\subsection{Width of a strip}
\label{subsect:length_segments}

Recall, see \cite[sect.~3.3]{Bu19}, that a 4-tuple
$p=(a,b)\in X^4$
with
$a=(x,y)$, $b=(u,z)$
is a {\em strip} if
$a$, $b$
are in the strong causal relation and the pairs
$(x,z)$, $(u,y)$
separate each other. Note that
$p'=(b,c)\in X^4$
with
$b=(x,u)$, $c=(y,z)$
is also a strip based on the same 4-tuple
$(x,y,u,z)\in X^4$.

Since the pairs
$a$, $b$
are in the strong causal relation, there is uniquely determined
common perpendicular
$s=(v,w)$
to
$a$, $b$.
We use notation
$p=(a,b,s)$
for a strip with common perpendicular
$s$.
Note that 
$s$
is uniquely determined by
$(a,b)$,
and we add 
$s$
to fix notation.

We define the width of the strip
$p$
as the length
$l=\width(p)$
of the segment
$x_su_s=y_sz_s\sub\h_s$
on the line
$\h_s$.

The following estimate has been obtained in \cite[Lemma~3.2]{Bu19}.

\begin{lem}\label{lem:length_preestimate} For any strip
$p=(a,b,s)$
we have
$$\width(p)\le 2\sqrt{\frac{|xu||yz|}{|xy||zu|}},$$
where
$a=(x,y)$, $b=(u,z)$. A similar estimate holds for the associated strip
$p'=(b,c,t)$,
where 
$t$
is common perpendicular to
$b=(x,u)$, $c=(y,z)$
$$\width(p')\le 2\sqrt{\frac{|xy||zu|}{|xu||yz|}},$$
in particular,
$\width(p)\cdot\width(p')\le 4$.
\end{lem}

\subsection{Diameter of elliptic quasi-lines in $\harm$}
\label{subsect:diam_quai-lines}

\begin{pro}\label{pro:diam_quasi-lines} There is a constant
$D>0$
such that for any involution
$\rho:X\to X$
without fixed points we have
$$\diam\wh e_\rho\le D,$$
where
$\wh e_\rho\sub\harm$
is the elliptic quasi-line associated with
$\rho$,
see Remark~\ref{rem:lift_quasi-line}, and
$\diam=\diam_\de$
is taken with respect to the distance
$\de$
in
$\harm$,
see sect.~\ref{subsect:dist_de}.
\end{pro}

In the proof, we use the construction from \cite[Lemma~3.3]{Bu18}, see sect.~\ref{subsect:zigzag_paths},
which gives a zz-path in
$\harm$
between given
$p$, $q\in\wh e_\rho$
consisting of 5 sides. We estimate the length of sides separately in 
Lemmas~\ref{lem:length_al_be_above}, \ref{lem:gamma_length}, \ref{lem:mu_estimates}, 
\ref{lem:nu_estimates}.

\begin{figure}[htbp]
\centering
\psfrag{t}{$t$}
\psfrag{s}{$s$}
\psfrag{u}{$u$}
\psfrag{z}{$z$}
\psfrag{x}{$x$}
\psfrag{y}{$y$}
\psfrag{c}{$c$}
\psfrag{d}{$d$}
\psfrag{e}{$e$}
\psfrag{f}{$f$}
\psfrag{v}{$v$}
\psfrag{al}{$\al$}
\psfrag{be}{$\be$}
\psfrag{ga}{$\ga$}
\includegraphics[width=0.6\columnwidth]{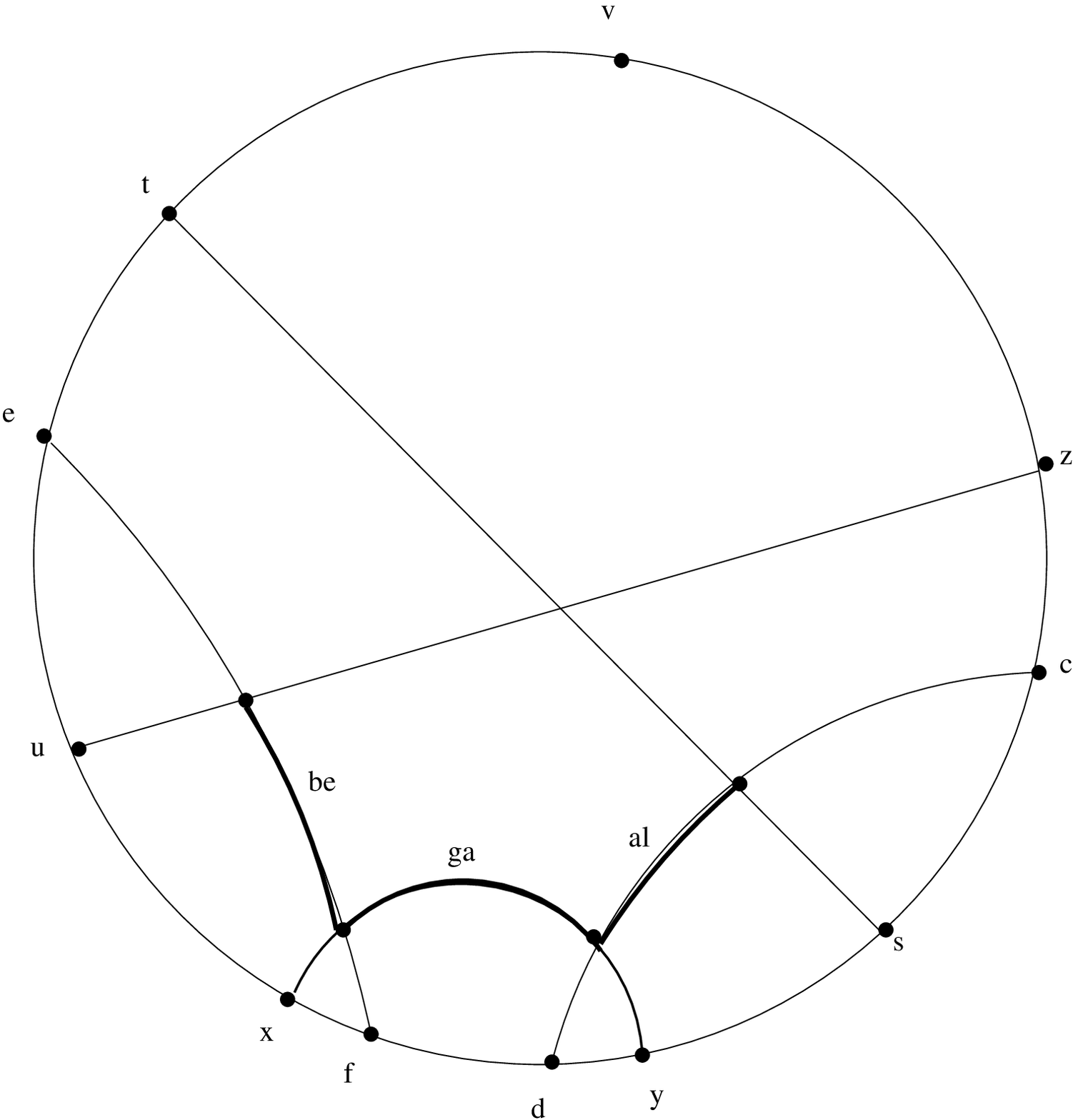}
\caption{}\label{fi:diam}
\end{figure}

Let
$(z,u)$, $(s,t)\in\ay$
be pairs which separate each other. They separate
$X$
into four open arcs. We choose one of them as follows.
Assume (without loss of generality) that
$|us||zt|\ge|zs||ut|$
(this does not depend of the choice of the metric from our M\"obius structure
$M$, 
in particular,
$|us|_t\ge|zs|_t$
in any metric 
$|\ |_t$
from
$M$
with infinitely remote point
$t$).

Then we take the arc 
$us\sub X$
between
$u$, $s$
that does not contain
$z$, $t$.
Next, we take a metric
$|\ |_t$
from
$M$
with infinitely remote point
$t$,
and take points
$x$, $y\in us$
(in the order
$uxys$)
such that
$|ux|_t=|xy|_t=|ys|_t=:h$.
It follows from continuity and monotonicity of the metric that
such points exist and they are uniquely determined.

Then the pairs
$(x,y)$, $(s,t)$
as well as the pairs
$(x,y)$, $(z,u)$
are in the strong causal relation, see sect.~\ref{subsect:lines}.
There are common perpendiculars
$(c,d)$
to the pairs
$(x,y)$, $(s,t)$,
and
$(e,f)$
to the pairs
$(x,y)$, $(z,u)$,
see Figure~\ref{fi:diam}.
These common perpendiculars
are uniquely determined, see sect.~\ref{subsect:lines}.

We estimate from above the length of the segments
$\al=x_{(c,d)}t_{(c,d)}=y_{(c,d)}s_{(c,d)}\sub\h_{(c,d)}$
and
$\be=x_{(e,f)}u_{(e,f)}=y_{(e,f)}z_{(e,f)}\sub\h_{(e,f)}$.

\begin{lem}\label{lem:length_al_be_above} In notations above we have
$|\al|\le 2$, $|\be|\le 4$.
\end{lem}

\begin{proof} For the strip
$p=(a,b,s)$, 
where
$a=(x,y)$, $b=(s,t)$, $s=(c,d)$,
we have
$|\al|=\width(p)$.
Lemma~\ref{lem:length_preestimate} gives
$|\al|\le 2\sqrt{\frac{|ys||xt|}{|xy||st|}}=2\sqrt{\frac{|ys|_t}{|xy|_t}}=2$.
 
Similarly, for the strip
$p'=(a,b',s')$,
where
$b'=(z,u)$, $s'=(e,f)$,
we have
$|\be|=\width(p')$.
Lemma~\ref{lem:length_preestimate} gives
$|\be|\le 2\sqrt{\frac{|xu||yz|}{|xy||uz}}=2\sqrt{\frac{|yz|_t}{|uz|_t}}$,
because
$|xu|_t=|xy|_t$.

Let
$v\in X$
be the point opposite to
$u$
with respect to the reflection 
$X\to X$
determined by the line
$\h_{(s,t)}$, 
i.e.
$u_{(s,t)}=v_{(s,t)}\in\h_{(s,t)}$.
Then
$|sv|_t=|us|_t\le 3h$
and
$v\not\in uz$
for the open arc 
$uz\sub X$,
that includes
$us$,
by the choice of the open arc 
$us\sub X$.
By the triange inequality and monotonicity
$|yz|_t\le|ys|_t+|sz|_t<h+|sv|_t\le 4h$, $|zu|_t>|xu|_t=h$.
Hence,
$|\be|\le 2\sqrt{\frac{4h}{h}}=4$.
\end{proof}

Next, we estimate from above the length
$|\ga|$
of the segment
$\ga=c_{(x,y)}e_{(x,y)}=d_{(x,y)}f_{(x,y)}\sub\h_{(x,y)}$
on the line 
$\h_{(x,y)}$.

\begin{lem}\label{lem:gamma_length} In notation above, we have
$|\ga|\le 6$.
\end{lem}

\begin{proof} Using notations above, we assume that the points
$d$, $f$
lie on the segment
$xy\sub X_t$.

We consider, first, the case when
$e\le t$,
that is,
$e=t$,
or
$e$
lies on the ray
$ut\sub X_t$.
In this case, the points
$d$, $f$
lies in the order
$xfdy$
on the segment
$xy\sub X_t$.
Indeed, the pairs
$(c,d)$, $(e,f)$
are in the strong causal relation being the perpendiculars to 
$(x,y)$.
Thus, the opposite assumption
$xdfy$
leads to the conclusion that the pairs
$(c,d)$
and
$(s,t)$
are in the strong causal relation. This contradicts the fact that
$(c,d)$
is a perpendicular to
$(s,t)$.

Now, we have
$$|\ga|=\ln\frac{|xd||yf|}{|xf||yd|}.$$
Note that
$|xd|_t<|xy|_t=h$
by monotonicity, because
$d$
lies in the interior of the segment
$xy\sub X_t$.
We have
$$|\al|=\ln\frac{|ds||cy|}{|dy||cs|}=\ln\frac{|cy|_t}{|dy|_t}$$
because
$|ds|_t=|cs|_t$.
By Lemma~\ref{lem:length_al_be_above} we have
$|\al|\le 2$.
Thus
$|dy|_t\ge|cy|_te^{-2}\ge|ys|_te^{-2}=he^{-2}$.
It follows that
$|xd|_t/|dy|_t\le h/(he^{-2})=e^2$.

Next, we estimate 
$|yf|_t/|xf|_t$
from above. Since
$yf\sub xy\sub X_t$,
we have
$|yf|_t<|xy|_t=h$
by monotonicity.

By Lemma~\ref{lem:length_al_be_above}, we have
$$e^{|\be|}=\frac{|uf||ex|}{|xf||eu|}\le e^4.$$
Hence
$|xf|\ge\frac{|uf||ex|}{e^4|eu|}$.
By monotonicity, we have
$|uf|_t>|ux|_t=h$
and
$|ex|_t>|eu|_t$,
where the last inequality uses the assumption
$e\le t$,
see beginning of the proof. Therefore,
$|xf|_t\ge h/e^4$,
and we conclude that
$|yf|_t/|xf|_t\le e^4$.
Hence,
$|\ga|\le\ln(e^2\cdot e^4)=6$.

Now, we consider the case
$e>t$,
that is, 
$e$
lies on the ray
$zt\sub X_t$.
In this case, we cannot garantee that the points
$d$, $f$
lies in the order
$xfdy$
on the segment
$xy\sub X_t$.
Thus we consider two subcases

(1) The points
$d$, $f$
lies in the order
$xfdy$
on the segment
$xy$.
We represent the length
$|\ga|$
as
$$e^{|\ga|}=\frac{|xc||ye|}{|xe||yc|},$$
and take a metric from
$M$
with the infinitely remote point
$u$.
We have
$xc\sub xe\sub X_u$,
thus
$|xc|_u/|xe|_u<1$,
and hence
$e^{|\ga|}\le |ye|_u/|yc|_u$.
Next, we use that
$$e^{|\be|}=\frac{|fz||ye|}{|fy||ze|}\le e^4$$
by Lemma~\ref{lem:length_al_be_above}. 
Since
$|fz|_u=|ze|_u$,
we obtain
$|ye|_u\le e^4|fy|_u$.
Since
$ys\sub yc\sub X_u$,
we have
$|ys|_u<|yc|_u$,
which gives
$e^{|\ga|}\le e^4|fy|_u/|ys|_u$.
Using the metric inversion, see (\ref{eq:metric_inversion}), we pass
to the metric with infinitely remote point
$t$
and use that
$|ys|_t=h$, $|fy|_t<|xy|_t=h$:

$$|fy|_u=\frac{|fy|_t}{|fu|_t|yu|_t}\le\frac{h}{|fu|_t|yu|_t}.$$
$$|ys|_u=\frac{|ys|_t}{|yu|_t|su|_t}=\frac{h}{|yu|_t|su|_t}.$$
Using that
$|fu|_t>|ux|_t=h$
by monotonicity and
$|su|_t\le 3h$
by the triange inequality, we finally obtain
$e^{|\ga|}\le e^4|fy|_u/|ys|_u\le e^4|su|_t/|fu|_t\le e^4\cdot3h/h=e^4\cdot 3$.
Hence,
$|\ga|\le 4+\ln 3$. 

(2) The points
$d$, $f$
lies in the order
$xdfy$
on the segment
$xy$.
Recall that the pairs
$(c,d)$
and
$(e,f)$
are in the strong causal relation, and the pairs
$(c,d)$, $(s,t)$
separate each other. Thus
$c$
lies on the ray
$et\sub X_t$
which does not contain
$d$.
Hence, this time we have
$xe\sub xc\sub X_t$
and
$$e^{|\ga|}=\frac{|xe||yc|}{|xc||ye|}.$$
By monotonicity,
$|xe|_t<|xc|_t$
and we conclude that
$e^{|\ga|}<|yc|_t/|ye|_t$.

To estimate
$|yc|_t$
from above, we use that
$$e^{|\al|}=\frac{|ds||yc|}{|dy||cs|}\le e^2$$
by Lemma~\ref{lem:length_al_be_above}. Since
$|ds|_t=|cs|_t$
and
$dy\sub xy\sub X_t$,
we obtain
$|yc|_t\le e^2|dy|_t\le e^2|xy|_t=e^2h$.
On the other hand,
$ys\sub ye\sub X_t$.
Thus
$|ye|_t>|ys|_t=h$
by monotonicity. Therefore
$e^{|\ga|}\le e^2$
and
$|\ga|\le 2$.
\end{proof}

Let
$p=((z,u),(z',u'))$, $q=((s,t),(s',t'))\in\wh e_\rho$
be given distinct harmonic pairs of pairs from
$\ay$.
Then the pairs
$(z,u)$, $(s,t)\in\ay$
separate each other being different members of the elliptic quasi-line in
$e_\rho\sub\ay$.
Assume as above (without loss of generality) that
$|us||zt|\ge|zs||ut|$.
Then we take the arc 
$us\sub X$
between
$u$, $s$
that does not contain
$z$, $t$.
We also assume that 
$t'$, $z'$
lie on the arc in
$X$
between
$s$, $t$ 
that contains
$su$.
\begin{rem}\label{rem:t_prime_in_su} In this case,
$sz'\sub st'\sub su\sub X_t$.
Indeed, since
$z=\rho(u)$, $s'=\rho(t')$,
the pairs of points
$(z,u)$, $(s',t')$
separate each other. Thus the opposite assumption
$u\in st'$
would imply
$|su|_t<|st'|_t=|ss'|_t<|sz|_t$,
a contradiction with our assumption
$|us||zt|\ge|zs||ut|$.
To show that 
$z'\in st'$,
we fix 
$q=((s,t),(s',t'))$
and move
$u$
from
$t'$
to
$t$
along the arc 
$t't\sub st$.
Then
$z'$
moves from
$s$
to
$t'$
along the arc 
$st'\sub st$.
Since
$u\in t't$
by the first part of the argument, we see that
$z'\in st'$.
\end{rem}

\begin{lem}\label{lem:mu_estimates} In notations above, assume that
$sy\sub st'\sub X_t$
(recall that 
$st'\sub su$,
see Remark~\ref{rem:t_prime_in_su}). Then
$|\mu|\le\ln 3$, 
where the segment
$\mu=d_{(s,t)}t'_{(s,t)}=c_{(s,t)}s'_{(s,t)}$
lies on the line
$h_{(s,t)}$.
\end{lem}

\begin{proof} We have
$$|\mu|=\left|\ln\frac{|sd||tt'|}{|st'||td|}\right|=\left|\ln\frac{|sd|_t}{|st'|_t}\right|.$$
Since
$sy\sub sd\sub su\sub X_t$,
we estimate
$h=|sy|_t\le|sd|_t\le|su|_t\le 3h$.
Since
$sy\sub st'\sub su$,
we estimate
$h=|sy|_t\le|st'|_t\le|su|_t\le 3h$. 
Thus
$|\mu|\le\ln 3$.
\end{proof}

\begin{lem}\label{lem:zz_prime_estimate} In notations above, we have 
$|zz'|_t\ge h$.
\end{lem}

\begin{proof} If
$sy\sub sz'\sub X_t$,
then
$|zz'|_t\ge|sz'|_\ge|sy|_t=h$.
Thus we assume that 
$sz'\sub sy$.
Then
$ux\sub uz'$
and hence
$|uz'|_t\ge|ux|_t=h$.

Since the pair of pairs
$p=((z,u),(z',u'))$
is harmonic, we have
$|zz'||uu'|=|zu'||z'u|$
in any metric of the M\"obius structure
$M$.
Note that 
$t$
lies on the arc in
$X$
between
$u$, $u'$
that does not contain
$z$, $z'$.
Thus we have
$|uu'|_t\le|uz'|_t+|z'z|_t+|zu'|_t$
by the triangle inequality in the metric
$|\ |_t$
with infinitely remote point
$t$.

Using notations
$|zu'|_t=:a$, $|z'u|_t=:b$, $|zz'|_t=\ep$,
we conclude that
$ab\le\ep(a+b+\ep)$.
Therefore,
$$\ep\ge\frac{a+b+\sqrt{(a+b)^2+4ab}}{2}\ge a+b>b.$$
But
$b=|z'u|_t\ge h$.
Hence
$|zz'|_t\ge h$
also in this case.
\end{proof}

\begin{lem}\label{lem:nu_estimates} In notations above, we have
$|\nu|\le\ln 18$, 
where the segment
$\nu=f_{(u,z)}z'_{(u,z)}=e_{(u,z)}u'_{(u,z)}$
lies on the line
$\h_{(u,z)}$.
\end{lem}

\begin{proof} We first show that
$|uz'|_t\ge h$.
If
$sz'\sub sy$,
then
$xy\sub uz'$
and hence
$h=|xy|_t\le |uz'|_t$.
Thus we assume that
$sy\sub sz'$.
Then
$|sy|_{s'}\le|sz'|_{s'}<|zz'|_{s'}$.
As in Lemma~\ref{lem:zz_prime_estimate} applied to a metric
$|\ |_{s'}$ 
with infinitely remote point
$s'$,
we obtain
$|uz'|_{s'}>|zz'|_{s'}>|sy|_{s'}$.
The metric inversion with respect to
$t$
gives
$$|uz'|_{s'}=\frac{|uz'|_t}{|us'|_t|z's'|_t};\
  |sy|_{s'}=\frac{|sy|_t}{|ss'|_t|ys'|_t}.$$
Using that
$|sy|_t=h$
and by monotonicity
$|ss'|_t<|z's'|_t$, $|ys'|_t<|us'|_t$,
we obtain
$|uz'|_t>h$.

Now, using monotonicity, Lemma~\ref{lem:zz_prime_estimate} and the first part of
the proof, we have the following two-sided estimates
for 
$|uz'|_t$, $|zf|_t$, $|uf|_t$
and
$|zz'|_t$:
 
$h\le|uz'|_t\le|us|_t\le 3h$,
$h=|sy|_t\le|zf|_t\le|zu|_t\le|uv|_t\le 6h$,
$h=|ux|_t\le|uf|_t\le|uy|_t\le 2h$,
$h\le|zz'|_t\le|uz|_t\le 6h$,
where the point
$v\in X_t$
is determined in Lemma~\ref{lem:length_al_be_above}.

Since
$$|\nu|=\left|\ln\frac{|uz'||zf|}{|uf||zz'|}\right|$$
for any metric from the M\"obius structure
$M$,
this gives
$|\nu|\le\ln 18$.
\end{proof}

Now, we estimate the length
of the zz-path
$\si=\mu\al\ga\be\nu$
in a particular case, when
$sy\sub st'\sub X_t$.

\begin{lem}\label{lem:zz_path_particular} In notations at the beginning of
the section, assume that
$sy\sub st'\sub X_t$
for the zz-path
$\si=\mu\al\ga\be\nu$
between
$p=((z,u),(z',u'))$
and
$q=((s,t),(s',t'))\in\wh e_\rho$.
Then
$|\si|\le D$
with
$D=12+\ln 54<16$.
\end{lem}

\begin{proof} We have
$|\al|\le 2$, $|\be|\le 4$
by Lemma~\ref{lem:length_al_be_above},
$|\ga|\le 6$
by Lemma~\ref{lem:gamma_length},
$|\mu|\le\ln 3$
by Lemma~\ref{lem:mu_estimates}
and
$|\nu|\le\ln 18$
by Lemma~\ref{lem:nu_estimates}. Note that the assumption
$sy\sub st'$
is only used in the estimate for 
$|\mu|$.
Thus
$|\si|\le|\mu|+|\al|+|\ga|+|\be|+|\nu|\le D$.
\end{proof}

In notations above, assume that
the harmonic pair
$q=((s,t),(s',t'))\in\wh e_\rho$
is fixed. Then the harmonic pair
$p=((z,u),(z',u'))\in\wh e_\rho$
is uniquely determined by the point
$u$
on the arc 
$tt'\sub X$
between
$t$, $t'$
that does not contain
$s$, $s'$
because
$z=\rho(u)$
and
$(z',u')\in\ay$
is determined by
$(z,u)$,
see Lemma~\ref{lem:harmonic_pairs_ellitic}. The point
$u$
in its own turn determines
$x$, $y\in us$.
The conclusion of Lemma~\ref{lem:zz_path_particular} holds
for 
$u\in tt'$
such that
$sy\sub st'$.
This gives an upper bound for the distance
$|us|_t$,
in particular,
$u$
is separated from
$t$.
Let
$u_0\in tt'$
by maximal with this property, i.e.
$y=t'$
for 
$y=y(u_0)$.

At the moment, we do not have a required estimate of
$|\si|$
for 
$u$
on the (open) arc 
$tu_0\sub tt'$.
To fill in this gap, we apply the same construction for 
$p=((z,u),(z',u'))$
and
$q'=j(q)=((s',t'),(s,t))$
assuming without loss of generality that
$|us'||zt'|\ge|zs'||ut'|$
and choosing the arc
$s'u\sub X$
between
$s'$, $u$
that does not contain
$z$, $t'$. 
Then
$u$
determines as above
$x'$, $y'\in s'u$
with
$|ux'|_{t'}=|x'y'|_{t'}=|y's'|_{t'}=:h'$.

Now, the conclusion of Lemma~\ref{lem:zz_path_particular} holds for 
$u\in tt'$
such that
$s'y'\sub s't\sub X_{t'}$.
Let
$u_1\in tt'$ 
be maximal with this property, i.e.
$y'=t$
for 
$y'=y'(u_1)$.
We show that the subarcs
$u_0t'$
and
$u_1t$
in
$tt'$
overlap. At this point, we need the condition
$\al\ge\sqrt{2}-1$
in Axiom~M($\al$).

\begin{lem}\label{lem:overlap} In notations above the arcs
$u_0t'$, $u_1t\sub X$
overlap,
$u_0t'\cap u_1t\not=\es$.
\end{lem}

\begin{proof} By the assumption on
$u_0$,
we have 
$h=|u_0x|_t=|xy|_t=|xt'|_t$.
Thus the pair
$((x,t),(u_0,t'))$
is harmonic. Then by Axiom~M($\al$)
$|u_0t'|_t\ge\sqrt{2}h$.
Taking the metric inversion, we obtain
$$|u_0t|_{t'}=\frac{|u_0t|_t}{|u_0t'|_t|tt'|_t}=\frac{1}{|u_0t'|_t}\le\frac{1}{\sqrt{2}h}.$$
Again, since
$|u_1x'|_{t'}=|x't|_{t'}=h'$,
the pair
$((x',t'),(u_1,t))$
is harmonic. By Axiom~M($\al$),
$|u_1t|_{t'}\ge\sqrt{2}h'$. 

We show that
$2hh'\ge 1$.
Note that
$h=|st'|_t=|ss'|_t$
by harmonicity of
$q$,
and
$h'=|s't|_{t'}=|ss'|_{t'}$
by harmonicity of
$q'=j(q)$.
Taking the metric inversion, we have
$$|ss'|_{t'}=\frac{|ss'|_t}{|st'|_t|s't'|_t}=\frac{1}{|s't'|_t}.$$
Since
$|s't'|_t\le|s's|_t+|st'|_t$
by the triange inequality, we see that
$|s't'|_t\le 2h$.
Then
$$hh'=\frac{|st'|_t}{|s't'|_t}\ge 1/2.$$
Therefore,
$2hh'\ge 1$.
Now,
$|u_1t|_{t'}\ge\sqrt{2}h'\ge\frac{1}{\sqrt{2}h}\ge|u_0t|_{t'}$.
Hence
$u_0t'\cap u_1t\not=\es$
by monotonicity.
\end{proof}

\begin{proof}[Proof of Proposition~\ref{pro:diam_quasi-lines}] We use notations
introduced above. For 
$p$, $q\in\wh e_\rho$,
$p=((z,u),(z',u'))$, $q=((s,t),(s',t'))$,
and
$x$, $y\in su\sub X_t$
with
$|ux|_t=|xy|_t=|ys|_t$,
if
$|ut'|_t\le|u_0t'|_t$,
then
$sy\sub st'$
and
$\de(p,q)\le D$
by Lemma~\ref{lem:zz_path_particular}. In particular, 
this condition is fulfilled for 
$p=q'=j(q)=((s',t'),(s,t))$
because then
$u=t'$.
Thus
$\de(q',q)\le D$.

In the opposite case,
$|ut'|_t>|u_0t'|_t$,
we have
$|ut|_{t'}\le|u_1t|_{t'}$
by Lemma~\ref{lem:overlap}. Hence
$\de(p,q')\le D$.
In this case,
$\de(p,q)\le\de(p,q')+\de(q',q)\le 2D$
by the triangle inequality. Therefore,
$\diam\wh e_\rho\le 2D$
with
$D<16$.
\end{proof}

\begin{cor}\label{cor:uniform_cobounded} The subspace
$\harm_\om\sub\harm$
is cobounded in
$\harm$
uniformly in
$\om\in X$,
that is, for any
$q\in\harm$, $\om\in X$
we have
$\dist_\de(q,\harm_\om)\le D$
for some universal constant
$D>0$.
\end{cor}

\begin{proof} We take one of two involutions associated with
$q\in\harm$,
see sect.~\ref{subsect:construction_harm_to_harm_omega}, and denote it by
$\rho$. 
Let
$\wh e_\rho\sub\harm$
be elliptic quasi-line associated with the involution
$\rho:X\to X$.
Then
$q\in\wh e_\rho$,
see sect.~\ref{subsect:involution_harm}, and by Lemma~\ref{lem:omega_projection},
$h_\om(q)\in\wh e_\rho\cap\harm_\om$.
Thus
$\dist_\de(q,\harm_\om)\le\de(q,h_\om(q))\le\diam_\de(\wh e_\rho)\le D$
by Proposition~\ref{pro:diam_quasi-lines}.
\end{proof}

\section{Hyperbolic approximation of $X_\om$}
\label{sect:hyperbolic_approximations}

A hyperbolic approximation is a kind of a hyperbolic cone over a metric
space, see \cite{BS07}. A specific feature of a hyperbolic approximation
of a metric space is that it is defined via families of metric balls in the space
in such a way to reflect their combinatorics.

\subsection{Definition}
\label{subsect:definition}

The set 
$\harm_\om$
of harmonic 4-tuples with common entry
$\om$
can be identified with the set of metric balls in
$X_\om$.
Indeed, every
$q=((a,b),(o,\om))\in\harm_\om$
determines the sphere
$S_r(o)=(a,b)$
because
$o$
is the midpoint between
$a$, $b$, $|ao|_\om=|ob|_\om=:r$,
and hence the ball 
$B_r(o)=\set{x\in X_\om}{$|ox|_\om\le r$}\sub X_\om$
with
$\d B_r(o)=S_r(o)$.

Vice versa, given a ball
$B_r(o)\sub X_\om$
of radius
$r>0$
centered at
$o$,
we have a 4-tuple
$q=((a,b),(o,\om))$,
where
$(a,b)=\d B_r(o)$,
which is harmonic, 
$q\in\harm_\om$,
because
$o$
is the midpoint between
$a$, $b$.

A (finite or infinite) sequence of spheres
$S_r(o_i)=(a_i,b_i)\sub X_\om$
is said to be a {\em harmonic chain of radius}
$r$
if the pair
$((a_i,b_i),(a_{i+1},b_{i+1}))$
is harmonic for every
$i$. 

Assuming that an orientation of (and hence an order on)
$X_\om$
is fixed, and that
$a_i<b_i$, $a_{i+1}<b_{i+1}$, $a_i<a_{i+1}<b_i$,
we observe that 
$b_{i+1}>b_i$
because the pairs
$(a_i,b_i)$, $(a_{i+1},b_{i+1})$
separate each other and
$a_{i+1}<b_{i+1}$.
Moreover,
$o_i<a_{i+1}$
since otherwise
$b_{i+1}=\om$
or
$b_{i+1}<a_{i+1}$.
Similarly,
$b_i<o_{i+1}$.

Speaking about harmonic chains of spheres, we mean that these
assumptions are always satisfied. Note that then the pairs
$(a_i,b_i)$, $(a_{i+2},b_{i+2})$
are in strong causal relation. Indeed, this is equivalent to
$b_i<a_{i+2}$,
which is fulfilled because otherwise
$a_{i+2}\le b_i$
and hence
$a_{i+2}<o_{i+1}$.
But this contradicts the inequality
$o_{i+1}<a_{i+2}$.

We fix 
$0<\si\le 1/24$
and for every
$k\in\Z$
let
$V_k\sub\harm_\om$
be an infinite in both directions harmonic chain of radius
$r=\si^k$.
We put
$V=\cup_{k\in\Z}V_k\sub\harm_\om$
and define a {\em harmonic hyperbolic} approximation 
$Z=Z(\si)$
of
$X_\om$
with parameter
$\si$
as a graph with the vertex set 
$V$.
We consider vertices in
$V$ 
as spheres (balls) of respective harmonic chains. For any
$v\in V$
we denote
$B(v)$
the respective ball in
$X_\om$.

Two vertices
$v$, $v'\in V$
are connected by an edge if and only is they lie on one and the same level 
$V_k$
and are in this case neighboring spheres,
$v=S_r(o_i)$, $v'=S_r(o_j)$
with
$|i-j|=1$
and
$r=\si^k$,
or 
$v\in V_k$, $v'\in V_l$
with
$|k-l|=1$
and in this case the respective ball with the larger level is contained in the respective
ball with the smaller lever, i.e.
$B_r(o_i)\sub B_{r'}(o_j)$
if
$r=\si^{k+1}$, $r'=\si^k$.

An edge
$vv'\sub Z$
is called {\em horizontal}, it its vertices lie on one and the same level,
$v$, $v'\in V_k$
for some
$k\in\Z$. Other edges are called {\em radial}. The level function
$\ell:V\to\Z$
is defined by
$l(v)=k$
for 
$v\in V_k$.
Since every level
$Z_k\sub Z$, $k\in\Z$,
is connected, the graph
$Z$
is connected. We endow
$Z$
with path metric assuming that the length of every edge is 1. We denote
by
$|vv'|$
the distance between points
$v$, $v'\in V$
in
$Z$.
Note that
$Z$
is geodesic because it is connected and distances between vertices take
integer values.

\subsection{ Geodesics in $Z$}
\label{subsect:geodesics_z}

The construction of the (harmonic) hyperbolic approximation
$Z$
here is slightly different from that in \cite{BS07}. Thus
we basically follow \cite[sect.~6.2]{BS07} with appropriate adaptation
of the arguments.

\begin{lem}\label{lem:ancestor} For every
$v\in V$
there is a vertex
$w\in V$
with
$\ell(w)=\ell(v)-1$
connected with any
$v'\in V$, $\ell(v')=\ell(v)$, $|vv'|\le 1$,
by a radial edge.
\end{lem}

\begin{proof} There are two neighbors
$v'$, $v''$
of
$v$
in
$Z$,
sitting on the same level as
$v$,
$|vv'|$, $|vv''|\le 1$.
One of them,
$v'$,
is on the left to
$v$,
the other one,
$v''$
is on the right to
$v$.
Let
$v'=(a',b')$, $v'=(a'',b'')$.
Then
$|a'b''|_\om\le 6r'$, 
where
$r'=\si^{k+1}$ 
for
$k+1=\ell(v)$.

On the other hand, for every neighboring
$w$, $w'\in V_k$,
$w=(c,d)$, $w'=(c',d')$,
the pair
$((c,d),(c',d'))$
is harmonic. Thus
$|c'd|_\om|cd'|_\om=|cc'|_\om|dd'|_\om$.
Hence
\begin{equation}\label{eq:harmonic_below}
|c'd|_\om=\frac{|cc'|_\om|dd'|_\om}{|cd'|_\om}\ge\frac{r}{4} 
\end{equation}
for 
$r=\si^k$
because
$|cc'|_\om\ge r$, $|dd'|_\om\ge r$
and
$|cd'|_\om\le 4r$. 

For the neighbors
$v'$, $v''$
of
$v$
we have
$v'\cup v\cup v''=a'b''\sub X_\om$.
Since
$\si=r'/r\le 1/24$,
we have
$|a'b''|_\om\le 6r'\le r/4$.
The balls
$\{w\in V_k\}$
cover
$X_\om$.
Assume that there is
$w\in V_k$
such that
$(v'\cup v\cup v'')\sub w$.
Then the vertices
$v$, $v'$, $v''\in Z$
are connected with
$w$
by radial edges. 

Otherwise 
$a'b''$
is covered by no
$w\in V_k$.
Then there are at most two neighboring 
$w=(c,d)$, $w'=(c',d')\in V_k$
which cover
$a'b''$, $a'b''\sub cd\cup c'd'$.
Assuming that 
$w$
is left to
$w'$,
we observe that the intersection
$w\cap w'=c'd$.
Since
$|a'b''|_\om\le |c'd|_\om$
by the estimate above, we see 
that
$a'b''$
is contained in one of
$w$, $w'$
in contradiction with our assumption. 
\end{proof}

\begin{lem}\label{lem:geodesics_in_z} Any vertices
$v$, $v'\in V$
can be connected in
$Z$
by a geodesic
$\ga$
which consists of at most two radial subsegments
$\ga'$, $\ga''\sub\ga$
and at most one horizontal edge between them. If there is such an edge, 
then it lies on the lowest level of the geodesic. Otherwise the unique
common vertex
$w$
of
$\ga'$, $\ga''$
is the lowest level vertex of
$\ga$.
\end{lem}

The proof proceeds exactly as in \cite[Lemma~6.2.6]{BS07}
using Lemma~\ref{lem:ancestor} and that fact that for any harmonic chain
$V_k$, $k\in\Z$
two balls
$v$, $v'\in V_k$
intersect if and only if they are neighboring in
$V_k$.
Thus we omit it.

\subsection{Hyperbolicity of $Z$}
\label{subsect:hypbolicity_z}

The Gromov product of
$v$, $v'$
with respect to
$u$
in a metric space
$Z$
is defined by
$$(v|v')_u=\frac{1}{2}(|vu|+|v'u|-|vv'|).$$
A metric space
$Z$
is said to be
$\de$-{\em hyperbolic}, $\de\ge 0$,
if for any
$v$, $v'$, $v''\in Z$
and a base point
$u\in Z$,
we have
$$(v|v')_u\ge\min\{(v|v'')_u,(v'|v'')_u\}-\de.$$
Now, we come back to our harmonic geodesic approximation
$Z$.

\begin{lem}\label{lem:adjacent_vertices} Assume that
$|vv'|\le 1$
for vertices
$v$, $v'\in Z$
of one and the same level,
$\ell(v)=\ell(v')$.
Then 
$|ww'|\le 1$
for any vertices
$w$, $w'\in Z$
adjacent to
$v$, $v'$
respectively and sitting one level below.
\end{lem}

\begin{proof} The balls
$B(w)$, $B(w')$
intersect because
$B(v)\sub B(w)$, $B(v')\sub B(w')$
and the balls
$B(v)$, $B(v')$
intersect.
Since
$w$, $w'$
are members of a harmonic chain, they are adjacent in
$Z$, $|ww'|\le 1$.
\end{proof}
 
From this we immediately obtain.

\begin{cor}\label{cor:radial_geodesics_common_ends} For any two radial geodesics
$\ga$, $\ga'\sub Z$
with common ends, the distance in
$Z$
between vertices of
$\ga$
and 
$\ga'$
of the same level is at most 1.
\end{cor}

It is convenient to use the following terminology. Let
$V'\sub V$
be a subset. A point
$u\in V$
is called a {\em cone point} for 
$V'$
if
$\ell(u)\le\inf_{v\in V'}\ell(v)$
and every 
$v\in V'$
is connected to
$u$
by a radial geodesic. A cone point of maximal level is called 
a {\em branch point} of
$V'$.

\begin{lem}\label{lem:cone_point} For any two points
$v$, $v'\in V$
there is cone point and, hence, a branch point. 
\end{lem}

\begin{proof} By Lemma~\ref{lem:geodesics_in_z}, 
$v$, $v'$
can be connected in
$Z$
by a geodesic 
$\ga$
which contains at most one horizontal edge. If there is no horizontal edge in
$\ga$,
then the lowest level point
$w$
of
$\ga$
is a branch point of
$v$, $v'$.
Otherwise, let
$uu'\sub\ga$
be the horizontal edge. It lies on the lowest level of
$\ga$.
Without loss of generality, we assume that
$vu$, $v'u'$
are radial geodesics.
By Lemma~\ref{lem:ancestor}, there is
$w\in V$
with
$\ell(w)=\ell(\ga)-1$
which is connected to 
$u$, $u'$
by radial edges. Taking concatenation
$vuw$, $v'u'w$
we see that 
$w$
connected to
$v$, $v'$
by radial geodesics. Hence,
$w$
is a cone point of
$v$, $v'$.
\end{proof}

Note that if
$u$
is a cone point of
$v$, $v'$
and
$w$
is their branch point, then
$(v|v')_u=|uw|$
in the case the geodesic
$vv'$
has no horizontal edge, and
$(v|v')_u=|uw|+1/2$
otherwise. In particular,
$|uw|\ge (v|v')_u-1/2$
is either case.

\begin{lem}\label{lem:cone_point_estimate} Let
$u\in V$
be a cone point of
$v$, $v'\in V$,
$\ga=uv$, $\ga'=uv'$
radial geodesics. Then for any
$y\in\ga$, $y'\in\ga'$
sitting one the same level
$\ell(y)=\ell(y')\le\ell(w)$,
where
$w$
is a branch point of
$v$, $v'$,
we have
$|yy'|\le 2$.
\end{lem}

\begin{proof} Concatenations
$vwu$, $v'wu$
are radial geodesics in
$Z$.
By Corollary~\ref{cor:radial_geodesics_common_ends}, we have
$|yy''|\le 1$
for 
$y\in\ga$, $y''\in vwu$
sitting on the same level,
$\ell(y)=\ell(y'')$,
and similarly
$|y'y''|\le 1$
for 
$y'\in\ga'$, $y''\in v'wu$
with
$\ell(y')=\ell(y'')$.
For
$\ell(y)=\ell(y')\le\ell(w)$
we can choose
$y''\in wu$
with
$\ell(y'')=\ell(y)=\ell(y')$,
and thus 
$|yy'|\le|yy''|+|y''y'|\le 2$.
\end{proof}

We need the following Proposition from \cite[Proposition~6.2.9]{BS07},
for which we give a different proof. 

\begin{lem}\label{lem:de_inequality_cone_point} Let
$v$, $v'$, $v''\in V$
and let
$w$, $w'$, $w''$
be branch points for the pairs of vertices
$\{v',v''\}$, $\{v,v''\}$
and
$\{v,v'\}$
respectively. Let
$u$
be a cone point of
$\{w,w',w''\}$.
Then
$$(v|v')_u\ge\min\{(v|v'')_u,(v'|v'')_u\}-\de$$
with
$\de=5/2$. 
\end{lem}

\begin{proof} We put
$t_0=\min\{|uw|,|uw'|\}$
and let
$\ga$, $\ga'$, $\ga''$
be radial geodesics between
$u$
and
$v$, $v'$, $v''$
respectively. Assume that
$y\in\ga$, $y'\in\ga'$, $y''\in\ga''$
satisfy
$|uy|=|uy'|=|uy''|=t_0$.
By Lemma~\ref{lem:cone_point_estimate} we have
$|yy''|$, $|y'y''|\le 2$.
Thus by the triangle inequality
$|yy'|\le 4$.
By monotonicity of the Gromov product
$$(v|v')_u\ge (y|y')_u=t_0-\frac{1}{2}|yy'|\ge t_0-2.$$ 
By the remark above
$t_0\ge\min\{(v|v'')_u,(v'|v'')_u\}-1/2$.
Hence, the claim.
\end{proof}

Using argument of \cite[Proposition~6.2.10]{BS07}, we obtain.

\begin{pro}\label{pro:hyperbolic_harmonic_approximation} Any hyperbolic harmonic
approximation
$Z$
of
$X_\om$
is a geodesic
$\de$-hyperbolic
space with
$\de=5$. 
\qed
\end{pro}

\section{$X_\om$ and $Z$ are quasi-isometric}
\label{sect:Xquasi-isometricZ}

Our aim is to show that for every
$\om\in X$
the space
$X_\om$
and its hyperbolic harmonic approximation
$Z=Z(\si)$
are quasi-isometric. Let
$V$
be the vertex set of
$Z$.
By definition, we have an inclusion
$f:V\hookrightarrow X_\om$.
We show that
$f$
is a quasi-isometry with respect to the metric on
$Z$
and the 
$\de$-metric 
on
$X_\om$.

\subsection{Estimates from above}
\label{subsect:estimates_above}

In this section we establish estimates from above, that is, we show that
there is a constant 
$D=D(\si)$
depending only on
$\si$ 
such that for every edge 
$vv'$
of
$Z$
we have 
$\de(v,v')\le D$.
For horizontal edges this is proven in Lemma~\ref{lem:equal_radius_harm},
and for vertical edges in Lemma~\ref{lem:different_levels_radius_unfixed}.

Fix 
$\om\in X$, $r>0$.
Then the sphere 
$S_r(o)\sub X_\om$
of radius
$r$
centered at
$o\in X_\om$
determines the harmonic pair
$((a,b),(o,\om))\in\harm$,
where
$S_r(o)=(a,b)$.

\begin{lem}\label{lem:equal_radius_harm} Fix 
$\om\in X$, $r>0$,
and consider two spheres 
$S_r(o)=(a,b)$, $S_r(o')=(a',b')$
in
$X_\om$
such that the pair of pairs
$((a,b),(a',b'))$
is harmonic. Then the 
$\de$-distance
between harmonic
$q=((a,b),(o,\om))$
and
$q'=((a',b'),(o',\om))$
is at most 
$2\ln 4$,
$\de(q,q')\le 2\ln 4$. 
\end{lem}

\begin{proof} We fix an orientation of
$X_\om$
and assume without loss of generality that the ordered pairs
$(a,b)$, $(a',b')$
agree with the orientation, and 
$a$
precedes
$b'$.
Note that
$b$
is not on the segment
$o'b'\sub X_\om$, $b\not\in o'b'$,
see sect.~\ref{subsect:definition}.
  
The harmonic pairs
$q=((a,b),(o,\om))$
and
$\wh q=((a,b),(a',b'))$
have the common axis
$(a,b)$.
Thus the distance
$l$
between
$q$, $\wh q$
along
$h_{(a,b)}$
is computed as
$$e^l=\frac{|aa'||ob|}{|ao||a'b|}=\frac{|aa'|_\om}{|a'b|_\om}$$
because
$|ao|_\om=r=|ob|_\om$.
Since
$\wh q$
is harmonic, we have
$|aa'||bb'|=|a'b||ab'|$.
Thus
$e^l=\frac{|ab'|_\om}{|bb'|_\om}$.
By the triange inequality and monotonicity,
$|ab'|_\om\le|ab|_\om+|bb'|_\om\le|ab|_\om+|a'b'|_\om\le 4r$.
By the remark above,
$|bb'|_\om\ge|o'b'|_\om=r$.
Therefore,
$l\le\ln 4$.
Similarly, 
$\wh q$
and
$q'$
have the common axis
$(a',b')$,
and 
$l'=|q'\wh q|\le\ln 4$.
Hence,
$\de(q,q')\le |q\wh q|+|q'\wh q|\le 2\ln 4$. 
\end{proof}

\begin{cor}\label{cor:horizont_edge_above} For every horizontal edge
$vv'\sub Z$
we have
$\de(v,v')\le C$
with
$C\le 2\ln 4$.
\end{cor}

\begin{proof} Indeed, the vertices
$v$, $v'$
of any horizontal edge in
$Z$
satisfy the condition of Lemma~\ref{lem:equal_radius_harm}.
\end{proof}

\begin{lem}\label{lem:containing_spheres} Fix 
$\om\in X$, $0<\si\le 1/24$, 
and consider two spheres
$S_r(o)=(a,b)$, $S_{r'}(o')=(a',b')$
in
$X_\om$,
where
$r=\si^k$, $r'=\si^{k+1}$
for some
$k\in\Z$,
such that
$o$
lies in the open interval
$(a'b')\sub X_\om$, $o\in (a'b')$.
Then the spheres
$(a,b)$, $(a',b')$
do not separate each other in
$X$.
Let
$h\sub\harm$
be the unique line that contains
$(a,b)$
and
$(a',b')$.
Then the distance
$l$
between
$(a,b)$ 
and 
$(a',b')$
along
$h$
is estimated above as
$l\le\sqrt{2/\si}$.
independent of
$k$.
\end{lem}

\begin{proof} To estimate
$l$
we use Lemma~\ref{lem:length_preestimate}. We assume as in the proof of
Lemma~\ref{lem:equal_radius_harm} that the ordered pairs
$(a,b)$, $(a',b')$
agree with a fixed orientation of
$X_\om$.

Since both
$o',o$
lies in the interval
$(a'b')\sub X_\om$,
we have
$|o'o|\le|a'b'|\le 2r'$.
Then
$|a'o|\le|a'o'|+|o'o|\le 3r'<r$
because
$\si\le 1/24$.
Hence
$a<a'$,
similarly
$b'<b$,
and the pairs
$(a,b)$, $(a',b')\sub X$
do not separate each other. Thus
$p=((a,b),(a',b'))$
is a strip. By Lemma~\ref{lem:length_preestimate} we have
$$l=\width(p)\le 2\sqrt{\frac{|aa'||bb'|}{|ab||a'b'|}}.$$
Since
$o\in(a'b')$,
it holds
$|aa'|_\om$, $|bb'|_\om\le r$.
Axiom~(M($\al$)) gives
$|ab|_\om\ge\sqrt{2}r$, $|a'b'|_\om\ge\sqrt{2}r'$.
Thus
$l\le 2\sqrt{r^2/2rr'}=\sqrt{2/\si}$.
\end{proof}

\begin{lem}\label{lem:different_levels_radius} Fix 
$\om\in X$, $0<\si\le 1/24$, 
and consider two spheres
$S_r(o)=(a,b)$, $S_{r'}(o')=(a',b')$
in
$X_\om$,
where
$r=\si^k$, $r'=\si^{k+1}$
for some
$k\in\Z$,
such that
$o$
lies in the open interval
$(a'b')\sub X_\om$, $o\in (a'b')$.
Then the 
$\de$-distance
between harmonic
$q=((a,b),(o,\om))$
and
$q'=((a',b'),(o',\om))$
is estimated above as
$\de(q,q')\le\sqrt{2/\si}+2\ln 3$
independent of
$k$. 
\end{lem}

\begin{proof} We fix an orientation and hence the respective order on
$X_\om$.
If
$o'=o$,
then
$q$, $q'$
lie on the line
$h_{(o,\om)}$,
and in this case
$\de(q,q')=|qq'|=\ln(r/r')=\ln(1/\si)<\sqrt{1/\si}$.
Thus we assume that
$o'\neq o$.

Without loss of generality, we assume that
$o'<o$
with respect to the order on
$X_\om$.
We also assume that
$a<b$, $a'<b'$.

As in Lemma~\ref{lem:containing_spheres}, the pairs
$(a,b)$
and
$(a',b')$
do not separate each other. Let
$(c,d)$
be the common perpendicular to
$(a,b)$
and
$(a',b')$,
$h=h_{(c,d)}\sub\harm$
the unique line containing
$(a,b)$
and
$(a',b')$.
Then we have a zz-path in
$\harm$
between
$q$, $q'$
which consists of 3 sides. 

First, one goes from
$q$
to
$\wh q=h_{(a,b)}\cap h$
along
$h_{(a,b)}$.
We denote the respective distance by
$m$.
 
Then one goes along
$h$
from
$\wh q$
to
$\wh q'=h\cap h_{(a',b')}$.
By Lemma~\ref{lem:containing_spheres}, the respective distance
$l$
is estimated above as
$l\le\sqrt{2/\si}$.

Finally, one goes from
$\wh q'$
along
$h_{(a',b')}$
to
$q'$.
We denote the respective distance by
$t$.
Thus we need to estimate above
$m$
and
$t$.

We assume without loss of generality that
$c\in(a'b')$.
Note that
$c\not\in o'o$,
since otherwise
$c$
is equal neither
$o$
nor 
$o'$
because
$o'\neq o$,
and 
$d$
must lie simultaneously left to
$a$
and right to
$b'$,
which is impossible.

We consider two cases
(1) $c<o'$
and
(2) $o<c$.

Case~(1). We have
$$e^m=\frac{|ao||bc|}{|ac||bo|}=\frac{|bc|_\om}{|ac|_\om}.$$
Using that
$|bc|_\om\le|ab|_\om\le 2r$
and
$|a'b'|_\om\le 2r'$,
we have
$|ac|_\om\ge|aa'|_\om\ge r-|a'b'|_\om\ge r-2r'$,
and obtain
$$e^m\le\frac{2r}{r-2r'}\le\frac{2}{1-2\si}\le 3.$$
On the other hand,
$$e^m=\frac{|a\om||bd|}{|ad||b\om|}=\frac{|bd|_\om}{|ad|_\om},$$
thus
$|bd|_\om/|ad|_\om\le 3$.

Now we compute 
$t$.
By the assumption
$c<o'<o$
we have
$d<a$.
Thus
$|b'd|_\om\le |bd|_\om$, $|a'd|_\om\ge|ad|_\om$
and we obtain
$$e^t=\frac{|a'\om||b'd|}{|a'd||b'\om|}=\frac{|b'd|_\om}{|a'd|_\om}\le\frac{|bd|_\om}{|ad|_\om}\le 3.$$
Thus
$t\le\ln 3$.

Case~(2). This is obtained similarly to case~(1) by interchanging
$a$, $b$
and
$a'$, $b'$.

Finally,
$\de(q,q')\le m+l+t\le\sqrt{2/\si}+2\ln 3$.
\end{proof}

\begin{lem}\label{lem:different_levels_radius_unfixed} Fix 
$\om\in X$, $0<\si\le 1/24$, 
and for a sphere
$S_r(o)=(a,b)\sub X_\om$
consider a maximal harmonic chain of spheres
$S_{r'}(o_i)=(a_i',b_i')\sub X_\om$, $i=1,\dots,n$,
that is contained in
$(a,b)$,
where
$r=\si^k$, $r'=\si^{k+1}$
for some
$k\in\Z$.
Then the 
$\de$-distance
between harmonic
$q=((a,b),(o,\om))$
and
$q_i'=((a_i',b_i'),(o_i',\om))$,
is estimated above as
$\de(q,q_i')\le c_1/\sqrt{\si}+c_2$
for every
$i=1,\dots,n$
independent of
$k$,
where
$c_1\le\sqrt{2}+4\ln 4$, $c_2=2\ln 3$.
\end{lem}

\begin{proof} The segments
$a_i'a_{i+1}'$, $i=1,\dots,n$
have disjoint interiors, and their union cover the union of spheres
$S_{r'}(o_i)$.
Thus
$$\sum_i|a_i'a_{i+1}'|\le|ab|\le 2r.$$
On the other hand,
$|a_i'a_{i+1}'|\ge|a_i'o_i'|=r'$
because
$o_i'$
lies in the interval
$a_i'a_{i+1}'$,
see sect.~\ref{subsect:definition}. Thus
$n\le 2r/r'=2/\si$.
There is
$j\in\{1,\dots,n\}$
such that
$o\in(a_j',b_j')$.
By Lemma~\ref{lem:different_levels_radius}, we have
$\de(q,q_j')\le\sqrt{2/\si}+2\ln 3$.

Using Lemma~\ref{lem:equal_radius_harm}, we obtain
$\de(q,q_i')\le\de(q,q_j')+\de(q_j',q_i')\le\sqrt{2/\si}+2\ln 3+2n\ln 4$
for every
$i=1,\dots,n$.
Therefore,
$\de(q,q_i')\le c_1/\sqrt{\si}+c_2$,
where
$c_1\le \sqrt{2}+4\ln 4$, $c_2=2\ln 3$.
\end{proof}

\begin{cor}\label{cor:vert_edge_above} For every vertical edge
$vv'\sub Z$
we have
$\de(v,v')\le C$
with
$C\le\sqrt{2/\si}+2\ln 3$.
\end{cor}

\begin{proof} Indeed, vertices
$v$, $v'$
of any vertical edge in
$Z$
satisfy the condition of Lemma~\ref{lem:different_levels_radius_unfixed}.
\end{proof}

\begin{cor}\label{cor:qi_above} For each pair of vertices
$v$, $v'\in V$
we have
$\de(v,v')\le C|vv'|_Z$
with
$C\le\sqrt{2/\si}+2\ln 3$.
\end{cor}

\begin{proof} Let
$\ga\sub Z$
be a geodesic between
$v$, $v'$, $\ga=v_0\dots v_n$
$v_0=v$, $v_n=v'$,
with edges
$v_iv_{i+1}$, $i=0,\dots,n-1$.
By definition, the length of
$\ga$
is the number of edges it consists,
$|vv'|_Z=|\ga|_Z=n$.
By Corollaries~\ref{cor:horizont_edge_above}, \ref{cor:vert_edge_above} we have
$\de(v_i,v_{i+1})\le C|v_iv_{i+1}|_Z=C$.
Thus
$\de(v,v')\le C|vv'|_Z$.
\end{proof}

\subsection{Estimates from below}
\label{subsect:estimates_below}

We fix an orientation of
$X$.
Then we have a respective order on each
$X_x$, $x\in X$,
induced by the orientation.

\begin{lem}\label{lem:distance_separated_spheres} Fix 
$\om\in X$, $r>0$,
and let
$S_r(o)=(a,b)$, $S_r(o')=(a',b')\sub X_\om$ 
be separated spheres with the order
$aba'b'$. 
Then the 
$\de$-distance
between harmonic pairs
$q=((a,b),(o,\om))$, $q'=((a',b'),(o',\om))\in\harm$,
is estimated above as
$\de(q,q')\le C(r,|ba'|_\om)$,
with
$C(r,|ba'|_\om)\le 4\ln\left(3\sqrt{\frac{r}{|ba'|_\om}}+\sqrt{\frac{|ba'|_\om}{r}}\right)$.
\end{lem}

\begin{proof} Since pairs
$(a,b)$, $(a',b')$
are in strong causal relation, there is a common perpendicular
$h=h_{(x,y)}$
to them. We assume without loss of generality that
$x<y$
with respect to our order on
$X_\om$.
This implies that
$o<x$
and
$y<o'$.

We have two harmonic
$p=((a,b),(x,y))$, $p'=((x,y),(a'b'))\in\harm$,
and we denote by
$\al$
the segment of
$h_{(a,b)}$
between
$q$
and
$p$,
by
$\ga$
the segment of
$h_{(x,y)}$
between
$p$
and
$p'$,
and by
$\be$
the segment of
$h_{(a',b')}$
between
$p'$
and
$q'$.
Then
$\si=\al\ga\be$
is a zz-path between
$q$, $q'$
which consists of three sides
$\al$, $\ga$, $\be$.
Since
$\de(q,q')\le|\si|$,
we estimate above
$|\si|=|\al|+|\be|+|\ga|$.

We have
$$e^{|\al|}=\frac{|ax|_\om|bo|_\om}{|ao|_\om|bx|_\om}=\frac{|ax|_\om}{|bx|_\om},$$
because
$|ao|_\om=r=|bo|_\om$.
Similarly,
$$e^{|\be|}=\frac{|a'o'|_\om|b'y|_\om}{|a'y|_\om|b'o'|_\om}=\frac{|b'y|_\om}{|a'y|_\om},$$
because
$|a'o'|_\om=r=|b'o'|_\om$.
Next
$$e^{|\ga|}=\frac{|xa'|_\om|by|_\om}{|xb|_\om|a'y|_\om}.$$
Harmonicity of
$p$
means that
$|bx||ay|=|ax||by|$,
and harmonicity of
$p'$
means that
$|a'y||b'x|=|a'x||b'y|$.
Using this, we obtain
$$L:=e^{|\al|+|\be|+|\ga|}=\frac{|b'x|_\om^2|ay|_\om^2}{|a'x|_\om|b'y|_\om|ax|_\om|by|_\om}.$$
Since
$bx\sub ob\sub X_\om$
and
$a'y\sub a'o'\sub X_\om$,
we have
$|bx|_\om\le r$, $|a'y|_\om\le r$
by monotonicity. Thus by the triange inequality
$|ay|_\om\le |ab|_\om+|ba'|_\om+|a'y|_\om\le 3r+|ba'|_\om$.
Similarly,
$|b'x|_\om\le 3r+|ba'|_\om$.

By monotonicity
$|xa'|_\om\ge|ba'|_\om$, $|by|_\om\ge|ba'|_\om$,
$|b'y|_\om\ge|o'b'|_\om=r$, $|ax|_\om\ge|ao|_\om=r$.
Therefore,
$$L\le\frac{(3r+|ba'|_\om)^4}{r^2|ba'|_\om^2},$$
and the required estimate follows.
\end{proof}

\begin{lem}\label{lem:de_estimate_harmonic_chain} Fix
$\om\in X$, $r>0$,
and let
$S_r(o_i)=(a_i,b_i)$, $i\in\Z$, 
be a harmonic chain in
$X_\om$.
Then for every sphere
$S_r(o)=(a,b)\sub X_\om$
we have
$\de(q,q_i)\le D=4\ln 160$,
where
$q=((a,b),(o,\om))$, $q_i=((a_i,b_i),(o_i,\om))\in\harm_\om$
with
$i\in\Z$
such that
$ab\cap a_ib_i\neq\es$.
\end{lem}

\begin{proof} If
$o=o_i$
for some
$i\in\Z$,
then
$q=q_i$,
and there is nothing to prove. Thus we assume that 
$o=o_i$
for no
$i\in Z$,
and furthermore we assume without loss of generality that
$i=0$,
and we have the following order 
$aa_0b$
of points on
$X_\om$.
 
Since
$a_0b_0\cap a_kb_k=\es$
for 
$|k|\ge 2$,
spheres
$S_r(o)$, $S_r(o_k)$
are separated. For 
$k\ge 4$,
the spheres
$S_r(o)$
and
$S_r(o_k)$
are separated by at least the sphere
$S_r(o_2)$,
Thus
$|ba_k|_\om\ge r$
in this case. On the other hand
$|ba_k|_\om\le |a_0a_k|_\om\le 2kr$
by the triangle inequality.

We let
$q_k=((a_k,b_k),(o_k,\om))$
be the respective harmonic pair. By Lemma~\ref{lem:distance_separated_spheres}, we have
$\de(q,q_k)\le C(r,|ba_k|_\om)$,
where
$$C(r,|ba_k|_\om)\le 4\ln\left(3\sqrt{\frac{r}{|ba_k|_\om}}+\sqrt{\frac{|ba_k|_\om}{r}}\right).$$
Thus
$\de(q,q_4)\le 4\ln(3+\sqrt{8})\le 4\ln 10$.
By Lemma~\ref{lem:equal_radius_harm}, 
$\de(q_k,q_0)\le 2|k|\ln 4$
for every
$k\in\Z$.
Therefore,
$\de(q,q_0)\le 4\ln 10+8\ln 4=4\ln 160=D$. 
\end{proof}

\begin{lem}\label{lem:cobouded_vertex} The set 
$V=V(\om,\si)$
is cobouded in
$\harm_\om$
uniformly in
$\om\in X$
with respect to the metric
$\de$,
that is,
$\de(p,V)\le D$
for every
$p\in\harm_\om$,
where
$D$
depends only on
$\si$.
\end{lem}

\begin{proof} Given
$p\in\harm_\om$, $p=((a,b),(o,\om))$, $(a,b)=S_r(o)$,
there is
$k\in\Z$
such that
$\si^{k+1}<r\le\si^k$.
We take
$q\in\harm_\om$, $q=((a',b'),(o,\om))$
with
$(a',b')=S_{\si^k}(o)$.
Then
$p$, $q$
lie on the line
$\h_{(o,\om)}$
and hence
$\de(p,q)\le|pq|=\ln\frac{\si^k}{r}\le\ln\frac{1}{\si}$
(in fact
$\de(p,q)=|pq|$
by Theorem~\ref{thm:de_metric_space}). By Lemma~\ref{lem:de_estimate_harmonic_chain},
there is
$q'\in V_k$
such that
$\de(q,q')\le D_1$
with
$D_1=4\ln 160$.
Thus
$\de(p,V)\le\de(p,q')\le\ln\frac{1}{\si}+D_1=:D$.
\end{proof}

Recall that by Lemma~\ref{lem:geodesics_in_z} any two vertices
$p$, $p'\in V$
are connected by a geodesic 
$\ga$
in
$Z$
which consists of at most two radial subsegments
$\ga'$, $\ga''\sub\ga$
and at most one horizontal edge
$h=qq'$
between them, possibly degenerated,
$q=q'$,
which lies on the lowest level of
$\ga$, $\ga=\ga'\cup h\cup\ga''$.
We assume that
$|\ga''|\le|\ga'|$
and consider two cases, the first is Lemma~\ref{lem:length_below_geod_hyp_approx},
the second one is Lemma~\ref{lem:dedistance_below}.

\begin{lem}\label{lem:length_below_geod_hyp_approx} 
Given vertices
$p$, $p'\in V$,
assume that
$|\ga''|\le 1$
for a geodesic
$\ga=\ga'\cup h\cup\ga''$
between
$p$, $p'$.
Then
$\de(p,p')\ge C|pp'|_Z-D$
for 
$C=\ln\frac{1}{\si}$
and a constant
$D\ge 0$
depending only on 
$\si$.
\end{lem}

\begin{proof} 
By our assumption,
$|\ga'|\ge|\ga|-2=|pp'|_Z-2$,
and
$\ga'\sub Z$
is a radial geodesic between harmonic
$p$
and
$q$
in
$X_\om$, $|\ga'|=|pq|_Z$,
where
$p=((a,b),(o,\om))=S_r(o)$, $q=((c,d),(o',\om))=S_{r'}(o')$,
$r=\si^l$, $r'=\si^k$.
For the levels
$l=\ell(p)$
and
$k=\ell(q)$
we have
$l>k$
and
$|\ga'|=l-k$.
The part
$e\cup\ga''$
of
$\ga$
consist of at most two edges between
$q$
and
$p'$,
one horizontal and one radial, thus
$\de(q,p')\le D_1$
by Corollaries~\ref{cor:horizont_edge_above}, \ref{cor:vert_edge_above},
with
$D_1\le\sqrt{2/\si}+2\ln 12$.

We take the sphere
$S_{r'}(o)=(a',b')\sub X_\om$,
and consider the harmonic
$\wh p=((a',b'),(o,\om))\in\harm_\om$.
Then by the triange inequality we have
$|\de(p,q)-\de(p,\wh p)|\le\de(q,\wh p)$.

Since
$q$
is a vertex of the hyperbolic approximation
$Z$,
the sphere
$S_{r'}(o)$
is a member of a harmonic chain. Since
$pq\sub Z$
is a radial geodesic segment,
$ab\sub cd\sub X_\om$.
By the choice of
$S_{r'}(o)$,
we have
$ab\sub a'b'$,
whence
$cd\cap a'b'\neq\es$.
Thus we can apply Lemma~\ref{lem:de_estimate_harmonic_chain} to 
$q$, $\wh p$,
and obtain
$\de(q,\wh p)\le D_2=4\ln 160$.
Therefore,
$\de(p,q)\ge\de(p,\wh p)-D_2$.

On the other hand,
$p$, $\wh p$
lie on a line in
$\harm_\om$,
thus
$|p\wh p|=\ln(r'/r)=(l-k)\ln\frac{1}{\si}$
because
$r'/r=1/\si^{l-k}$.
By Theorem~\ref{thm:de_metric_space},
$\de(p,\wh p)=|p\wh p|$.
Furthermore,
$|pq|_Z=l-k$
because
$pq\sub Z$
is a radial geodesic segment. Therefore,
$\de(p,q)\ge C|pq|_Z-D_2$
with
$C=\ln(1/\si)$.
Finally,
$\de(p,p')\ge\de(p,q)-\de(q,p')\ge C|pq|_Z-(D_1+D_2)\ge C(|pp'|_Z-2)-(D_1+D_2)=
C|pp'|_Z-D$
with
$D=2C+D_1+D_2$.
\end{proof}

\begin{lem}\label{lem:dedistance_below} Given vertices
$p$, $p'\in V$,
assume that
$|\ga''|\ge 2$
for a geodesic
$\ga=\ga'\cup h\cup\ga''$
between
$p$, $p'$.
Then
$\de(p,p')\ge C|pp'|_Z-D$
with
$C=\frac{1}{2}\ln\frac{1}{\si}$ 
and 
$D$
depending only on
$\si$.
\end{lem}

\begin{proof} As in Lemma~\ref{lem:length_below_geod_hyp_approx},
$\ga'\sub Z$
is a radial geodesic between harmonic
$p$
and
$q$
in
$X_\om$, $|\ga'|=|pq|_Z=l-k$,
where
$\ell(p)=l$, $\ell(q)=k$.

Then
$k$
is the level of
$qq'$, $k=\ell(q)=\ell(q')$
and
$\ga''\sub Z$
is a radial geodesic between harmonic
$q'$
and
$p'$
in
$X_\om$, $|\ga''|=|p'q'|_Z=l'-k$
where
$\ell(p')=l'$.
By our assumption
$|\ga'|\ge|\ga''|$. 
Thus
$l\ge l'$,
and
$|pp'|_Z=|\ga|\le|\ga'|+|\ga''|+1\le 2|\ga'|+1=2(l-k)+1$.

Let
$S$
be a zz-path in
$\harm$
between
$p$, $p'$
that approximates the distance
$\de(p,p')$, $\de(p,p')\ge |S|-\ep$
for some
$\ep>0$.
We fix an involution
$\rho:X\to X$
associated with 
$p'=((a',b'),t')$, $t'=(o',\om)$,
see sect.~\ref{subsect:involution_harm}, and let
$e=e_\rho$
be the respective elliptic quasi-line. By Lemma~\ref{lem:harmonic_pairs_ellitic},
there is a unique
$s\in e$
such that the pair
$\wh q=(s,t)$
is harmonic, where
$p=((a,b),t)$, $t=(o,\om)$.
Again, by Lemma~\ref{lem:harmonic_pairs_ellitic}, there is a unique
$t''\in e$
such that the pair
$(s,t'')$
is harmonic. Thus
$q''=(s,t'')\in\wh e$
as well as
$p'\in\wh e$
by definition of
$e$.
By Proposition~\ref{pro:diam_quasi-lines},
$\de(p',q'')\le D_0$
for some universal constant
$D_0<16$.
Hence, there is a zz-path
$S'$
between
$p'$
and
$q''$
with
$|S'|\le D_0+\ep$.

Note that
$t$, $t''$
lie on the line
$\h_s$.
Let
$S''$
be a zz-path between
$q''=(s,t'')$
and
$\wh q=(s,t)$
which consists of one side,
$S''\sub\h_s$.
Then
$p=((a,b),t)$
and
$\wh q$
lie on the line
$\h_t$.
Thus the concatenation
$\wh S:=S\ast S'\ast S''\ast \wh qp$
is a closed zz-path in
$\harm$.
We apply \cite[Proposition~6.1]{Bu18} to conclude
$|S|+|S'|+|S''|>|p\wh q|$.
The projection
$\pr_t:S\ast S'\ast S''\to\h_t$
does not increase distances, see \cite[Lemma~5.5 and Proposition~6.1]{Bu18},
and
$|\pr_t(S'')|=0$
because
$\pr_t(S'')=\wh q$.
Therefore,
$|S|\ge |p\wh q|-(D_0+\ep)$.

By definition of
$e$,
we have
$t'=(o',\om)\in e$.
We denote
$s=(z,u)$. 
Since
$(s,t)$
is harmonic, we have
$|zo|_\om=|ou|_\om$.
We assume without loss of generality that
$o<o'$, $z<o<u$
with respect to our fixed order on
$X_\om$.
Since 
$s$, $t'\in e$,
the pairs
$s=(z,u)$
and
$t'=(o',\om)$
separate each other, see Lemma~\ref{lem:separate}. Hence,
$|ou|_\om>|oo'|_\om$.

We denote by
$p_n$
the vertex of
$\ga'$
on the level
$n$, $\ell(p_n)=n$, $k\le n\le l$,
and similarly by
$p_n'$
the vertex of
$\ga''$
on the level
$n$, $\ell(p_n')=n$, $k\le n\le l'$.
Denote by
$\al_n$
the curve in
$Z$
between
$p_n$
and
$p_n'$
consisting horizontal edges. By the assumption
$|\ga''|\ge 2$,
thus there is a vertex
$p_n'\in\ga''$
with
$n=k+2$.
Note that 
$|\al_{k+2}|\ge 4$
because otherwise we can shorten the geodesic
$\ga$
between
$p$
and
$p'$.
Therefore, there is an edge
$vv'\sub\al_{k+2}$
with vertices
$v$, $v'$
different from the ends 
$p_{k+2}$, $p_{k+2}'$
of
$\al_{k+2}$.
Thus the intersection
$B_v\cap B_{v'}$
misses the balls
$B_{p_{k+2}}$
and
$B_{p_{k+2}'}$
by properties of harmonic chains. Here
$B_v\sub X_\om$
is the ball corresponding to the vertex
$v\in V$.

Since
$\ga'$, $\ga''\sub Z$
are radial geodesics, we have
$B_p\sub B_{p_{k+2}}$, $B_{p'}\sub B_{p_{k+2}'}$
for respective balls in
$X_\om$.
Recall that
$o$
is the center of
$B_p$,
and
$o'$
the center of
$B_{p'}$.
It follows that the intersection
$B_v\cap B_{v'}$
is a segment on
$X_\om$
lying inside of the segment
$oo'\sub X_\om$.
By inequality~(\ref{eq:harmonic_below}),
$|B_v\cap V_{v'}|\ge r/4$
for 
$r=\si^{k+2}$,
and we obtain
$|oo'|_\om\ge\si^{k+2}/4$.
Thus
$$|p\wh q|=\ln\frac{|ou|_\om}{\si^l}\ge\ln\frac{|oo'|_\om}{\si^l}\ge\ln\frac{\si^{k+2}}{\si^l}=
(l-k-2)\ln\frac{1}{\si}.$$
Since
$|\ga|\le 2(l-k)+1$,
we have
$|p\wh q|\ge|\ga|/2\cdot\ln\frac{1}{\si}-D_1$
with
$D_1=\frac{5}{2}\ln\frac{1}{\si}$.
Therefore,
$$|S|\ge|p\wh q|-(D_0+\ep)\ge C|\ga|-(D_0+D_1+\ep),$$
where
$C=\frac{1}{2}\ln\frac{1}{\si}$.
Finally, we conclude
$\de(p,p')\ge C|pp'|_Z-D$,
where
$D=D_0+D_1$.
\end{proof}

\begin{pro}\label{pro:inclusion_quasi-isometry} The inclusion
$f:V\hookrightarrow\harm_\om$
is a quasi-isometry with respect to the metric on
$Z$
and
$\de$-metric 
on
$\harm_\om$.
\end{pro}

\begin{proof} By Corollary~\ref{cor:qi_above} we have
$\de(v,v')\le C|vv'|_Z$
for every pair vertices
$v$, $v'\in V$,
where the constant
$C$
depends only on
$\si$.
By Lemmas~\ref{lem:length_below_geod_hyp_approx} and \ref{lem:dedistance_below}
we have
$\de(v,v')\ge C|vv'|_Z-D$
for every pair vertices
$v$, $v'\in V$,
where the constants
$C$, $D$
depend only on
$\si$.
Thus the map 
$f$
is quasi-isometric. By Lemma~\ref{lem:cobouded_vertex}, the set 
$V$
is cobouded in
$\harm_\om$.
Thus
$f$
is quasi-isometry.
\end{proof}

\begin{pro}\label{pro:de_metric_space_hyp} Assume that a M\"obius structure 
$M$
on
$X=S^1$
is strictly monotone, i.e., it satisfies axioms~(T), (M($\al$)), (P),
and satisfies Increment axiom. Then
$(\harm,\de)$
is a complete, proper, hyperbolic geodesic metric space with
$\de$-metric
topology coinciding with that induced from
$X^4$. 
\end{pro}

\begin{proof} By Theorem~\ref{thm:de_metric_space},
$(\harm,\de)$
is a complete, proper, geodesic metric space with
$\de$-metric
topology coinciding with that induced from
$X^4$.
By Corollary~\ref{cor:uniform_cobounded}, any its subset
$\harm_\om$, $\om\in X$,
is quasi-isometric
$(\harm,\de)$.
Using Proposition~\ref{pro:inclusion_quasi-isometry}, we see that
$(\harm_\om,\de)$
is quasi-isometric to its hyperbolic approximation
$Z=Z(\om,\si)$. Thus
$(\harm,\de)$
is quasi-isometric to
$Z$.
By Proposition~\ref{pro:hyperbolic_harmonic_approximation},
$Z$
is hyperbolic. Since both spaces
$(\harm,\de)$
and
$Z$
are geodesic, the space
$(\harm,\de)$
is hyperbolic. 
\end{proof}

\begin{proof}[Proof of Theorem~\ref{thm:main}] We define
$Y=(\harm,\de)$.
By Proposition~\ref{pro:de_metric_space_hyp},
$Y$
is a complete, proper, hyperbolic geodesic metric space. We clearly
have
$\di\harm_\om=X_\om$
for every
$\om\in X$.
Since
$\harm_\om$
is cobouded in
$Y$,
we have
$\di Y=\harm_\om\cup\{\om\}=X=S^1$.
The fact that the induced M\"obius structure
$M_Y$
on
$X$
is isomorphic to
$M$
is tautological because all of the geometry of
$Y$
including
$Y$
itself is determined via
$M$.
In particular, given two points
$x$, $x'\in\di Y$,
we take
$\om\in\di Y$
different from
$x$, $x'$.
Then
$x$, $x'\in X_\om$,
and we consider the line
$\h=\h_{(x,\om)}\sub\harm_\om\sub Y$.
Furthermore, we fix 
$y\in X_\om$, $y\neq x$,
and observe that there are points
$p$, $q\in\h$
such that
$x\in p$, $y\in q$.
Then
$|xx'|_\om=\be e^{\pm|pq|}$
for some fixed constant
$\be(=|xy|_\om)$.
In other words, the metric of
$X_\om$
is recovered from the geometry of
$Y$.
\end{proof}

\end{document}